\documentclass[reqno, 11pt]{amsart}
\usepackage[percent]{overpic}
\usepackage{calc,graphicx,amsfonts,amsthm,amscd,epsfig,psfrag,amsmath,amssymb,enumerate,dsfont,mathrsfs,paralist}

\usepackage{subfig} 
\usepackage{pdfsync}
\usepackage{stmaryrd} 
\usepackage[numeric,initials,nobysame]{amsrefs}
\usepackage{marginnote}
\usepackage[top=3.5cm, bottom=3.5cm, outer=3.5cm, inner=3.5cm, heightrounded, marginparwidth=2.2cm, marginparsep=0.5cm]{geometry}

\usepackage[american]{babel}
\usepackage{amsmath}
\usepackage{dsfont,mathtools,amssymb}
\usepackage{color}

\mathtoolsset{showonlyrefs,showmanualtags}
\usepackage[showlabels,sections,floats,textmath,displaymath]{}

\usepackage[showlabels,sections,floats,textmath,displaymath]{}

\reversemarginpar
\newlength\fullwidth
\numberwithin{equation}{section}

\DeclareMathSymbol{\leqslant}{\mathalpha}{AMSa}{"36} 
\DeclareMathSymbol{\geqslant}{\mathalpha}{AMSa}{"3E} 
\DeclareMathSymbol{\eset}{\mathalpha}{AMSb}{"3F}     
\renewcommand{\leq}{\;\leqslant\;}                   
\renewcommand{\geq}{\;\geqslant\;}                   
\newcommand{\maxtwo}[2]{\max_{\substack{#1 \\ #2}}} 
\newcommand{\suptwo}[2]{\sup_{\substack{#1 \\ #2}}} 
\renewcommand{\b}{\beta}

\def\1{\ifmmode {1\hskip -3pt \rm{I}} \else {\hbox {$1\hskip -3pt \rm{I}$}}\fi}

\newcommand{\var}{\operatorname{Var}}

\newcommand{\si}{\sigma }

\newcommand{\D}{\Delta}

\renewcommand{\b}{\beta}
\renewcommand{\l}{\lambda}

\renewcommand{\l}{\lambda}
\renewcommand{\a}{\alpha}
\renewcommand{\d}{\delta}
\renewcommand{\t}{\tau}

\newcommand{\g}{\gamma}

\newcommand{\e}{\varepsilon}
\renewcommand{\r}{\rho}

\renewcommand{\O}{\Omega}
\renewcommand{\epsilon}{\varepsilon}

\newcommand{\tc}{\thinspace |\thinspace}

\newtheorem{theorem}{Theorem}[section]
\newtheorem{maintheorem}{Theorem}
\newtheorem{lemma}[theorem]{Lemma}
\newtheorem{proposition}[theorem]{Proposition}
\newtheorem{corollary}[theorem]{Corollary}
\newtheorem{maincorollary}{Corollary}
\newtheorem{remark}[theorem]{Remark}

\newtheorem{definition}[theorem]{Definition}

\newcommand{\cA}{\ensuremath{\mathcal A}}

\newcommand{\cG}{\ensuremath{\mathcal G}}
\newcommand{\cH}{\ensuremath{\mathcal H}}

\newcommand{\cL}{\ensuremath{\mathcal L}}

\newcommand{\cO}{\ensuremath{\mathcal O}}
\newcommand{\cP}{\ensuremath{\mathcal P}}

\newcommand{\cZ}{\ensuremath{\mathcal Z}}

\newcommand{\bbE}{{\ensuremath{\mathbb E}} }

\newcommand{\bbL}{{\ensuremath{\mathbb L}} }

\newcommand{\bbN}{{\ensuremath{\mathbb N}} }

\newcommand{\bbP}{{\ensuremath{\mathbb P}} }

\newcommand{\bbR}{{\ensuremath{\mathbb R}} }

\newcommand{\bbZ}{{\ensuremath{\mathbb Z}} }

\newcommand{\wt}{\widetilde }

\newcommand{\IND}{{\bf 1}}

\newcommand{\be}{\boldsymbol \e}

\let\a=\alpha \let\b=\beta   \let\d=\delta  \let\e=\varepsilon
 \let\g=\gamma     \let\k=\kappa  \let\l=\lambda
            
\let\r=\rho  \let\s=\sigma \let\t=\tau   
  
\let\D=\Delta      
\let\O=\Omega

\renewcommand{\le}{\leq}

\title{Multi-level pinning problems for random walks and self-avoiding lattice
  paths}
\author{Pietro Caputo, Fabio Martinelli and Fabio Lucio Toninelli}
\address{Dipartimento di Matematica e Fisica, Universit\`a Roma
  Tre, Largo S. Murialdo 1, 00146 Roma, Italy}
\email{caputo@mat.uniroma3.it, martin@mat.uniroma3.it}

\address{Universit\'e de Lyon, CNRS and Institut Camille Jordan, Universit\'e Lyon 1,
    43 bd
 du 11 novembre 1918, 69622 Villeurbanne, France}
\email{toninelli@math.univ-lyon1.fr}

\thanks{This work was partially supported by the Marie Curie
  IEF Action ``DMCP- Dimers, Markov chains and Critical Phenomena'',
  grant agreement n.  621894}
\begin{document}
\begin{abstract}
We consider a generalization of the classical pinning problem for
integer-valued random walks conditioned to stay non-negative. More
specifically,  we take 
pinning potentials of the form $\sum_{j\geq 0}\e_j N_j$, where $N_j$ is
the number of visits to the state $j$ and $\{\e_j\}$ is
a non-negative sequence. Partly motivated by similar problems for low-temperature
contour models in statistical physics, we aim at finding a sharp characterization of the
threshold of the wetting transition, especially in the regime where
the variance $\s^2$ of the single step of the random walk is small.
Our main result says that, for natural choices of the pinning sequence $\{\e_j\}$,
localization (respectively delocalization)
occurs if $\s^{-2}\sum_{ j\geq0}(j+1)\e_j\geq\d^{-1}$ (respectively $\le \d$),
for some universal $\d <1$. Our
finding is reminiscent of the classical Bargmann-Jost-Pais criteria
for the absence of bound states for the radial Schr\"odinger
equation. The core of the proof is a recursive argument to bound the free energy
of the model. Our approach is rather robust, which allows us to obtain similar results in the 
case where the random walk trajectory is replaced by a self-avoiding path $\g$
in $\bbZ^2$ 
with weight $\exp(-\beta |\g|)$, $|\g|$ being the
length of the path and $\b>0$ a large enough parameter. This generalization is directly relevant for
applications to the above mentioned contour models.       
\end{abstract}

\maketitle

\section{Introduction and motivations}
Consider a one-dimensional integer-valued symmetric random walk starting at 
zero, conditioned to
stay non-negative. If the walk
has a reward $\e>0$ for each return to zero, it is a classical fact
that there exists  a critical
value $\e_c$ such that for $\e>\e_c$ the random walk has a positive
density of returns to the origin while for $\e<\e_c$  entropic
repulsion prevails and the density of returns is zero; see e.g.\ \cite{Giacomin}
and references therein. This is often called a wetting transition \cite{Fisher}. The critical parameter $\e_c$ depends
crucially on the single-step variance $\s^2$; in simple examples such
as the symmetric walk with increments in $\{-1,0,+1\}$ one finds that
$\e_c$ scales linearly in $\s^2$ as $\s^2\to 0$ \cite{IsoYos}.

In this work we consider a natural generalization where the 
pinning at the origin is replaced by a long range pinning potential
 $\be=\{\e_j\}_{j\geq 0}$, where $\e_j\geq 0$ is the reward for a visit to
the state $j\geq 0$. To be specific, for a trajectory $\g$ of length
$L$, define 
\begin{align}\label{model}
\Phi(\g)= \sum_{j=0}^\infty\e_j
N_j(\g),
\end{align} where $N_j(\g)$ is the number of visits to state $j$. Define also the free energy
\[  
f(\be)= \lim_{L\to \infty}\frac 1L
\log\bbE^+_{0,L}\big(e^{\Phi}\big),
\]
where $\bbE_{0,L}^+(\cdot)$ stands for the expectation w.r.t.\ to the path
measure conditioned to $\g\geq 0$ and $\g_0=\g_L=0$. The existence of
the limit follows by sub-additivity. 
With this notation the localized (resp.\ delocalized)
phase is characterized by $f(\be)>0$ (resp.\ $f(\be)=0$). Under mild assumptions on the random walk kernel
and on the pinning sequence $\be$, we prove that the wetting
transition occurs at a critical value
$\rho_c$ of
the ratio 
\[
\rho:=\frac{1}{\s^2}\sum_{j=0}^\infty (j+1)\e_j 
\] 
and that $\rho_c\in (a,b)$ for
universal constants $a,b>0$. 

 As far as we know, this is the first analysis of the wetting transition for a multi-level pinning problem of the general form \eqref{model}. We refer to \cites{Caravennaet,Sohier} for previous  studies of certain specific models of random walks with pinning on several layers.  
It is interesting  to note the analogy between our condition for
delocalization and  the classical Bargmann-Jost-Pais
\cites{Jost,Bargman,Solomyak} criteria
for the absence of bound states for the radial Schroedinger
equation; cf.\ Remark \ref{Jost} below for a discussion of this point.

The most challenging part of the proof is to show delocalization for $\rho$
small. That requires establishing an upper bound on the
partition function 
\[
Z^{\Phi,+}_{0,L}:= \sum_{\g\geq 0} w(\g)e^{\Phi(\g)},
\]
where the sum runs over non-negative trajectories returning to the origin at time
$L$ and $w(\g)$ is the probability of $\g$. 
Using the strategy outlined below we prove that 
\begin{equation}\label{xx}
Z^{\Phi,+}_{0,L}\leq C\,Z_{0,L},
\end{equation}
where $C$ is a universal constant and $Z_{0,L}= \sum_{\g} w(\g)$, the sum being over all  trajectories returning to the origin at time $L$. Clearly,
\[
\bbE^+_{0,L}\big(e^{\Phi}\big) = \frac{Z^{\Phi,+}_{0,L}}{Z_{0,L}}\, \bbP_{0,L}(\g\geq 0)^{-1}
\]
where $\bbP_{0,L}(\g\geq 0)$ is the probability that a path returning to the origin after $L$ steps remains non-negative. Well known bounds show that $\bbP_{0,L}(\g\geq 0)^{-1}=O(L)$, so that the estimate \eqref{xx} establishes delocalization. 

To prove \eqref{xx} we argue as follows. The first step is to
decouple the problem into a collection of independent 
pinning problems, one for each height level $j=0,1,\dots$. 
More precisely, let $\r_j:= (\rho \si^2)^{-1}(j+1)\e_j$ so that
$\sum_{j\geq 0}\r_j=1$. Then, Jensen's inequality implies
that
\[
 Z^{\Phi,+}_{0,L} =\sum_{\g\geq 0} w(\g)e^{\sum_{j=0}^\infty\e_j N_j(\g)}\le \sum_{j=0}^\infty \r_j  
 Z_{0,L}^{+,\k_j}\,,
\] 
where $ Z_{0,L}^{+,\k_j}=\sum_{\g\geq 0} w(\g)   e^{\rho
   \si^2N_j(\g)/(j+1)}$
is the partition function of a
random walk returning to the origin after $L$ steps, pinned at
height $j$ with pinning strength $\k_j= \rho
   \si^2/(j+1)$. 

The next step is to show that, if the parameter $\rho$ is small enough, then 
\emph{uniformly} in the law of the random walk and in the height $j$,
one has $Z_{0,L}^{+,\k_j}\leq C Z_{0,L} $. Using $\sum_{j}\rho_j=1$, this bound implies \eqref{xx}.

The main idea of the proof goes  as follows. With a natural 
inductive argument we 
show that, for all $L$ larger than a
critical ``diffusive'' scale $L_c(j)\sim (j+1)^2/\s^2$, one has 
$Z_{0,L}^{+,\k_j}\le C Z_{0,L}$, provided that the same holds for $L=L_c(j)$. 
The base case of the induction is solved by a fine
analysis based on careful local limit theorem estimates. It is
only at this stage, that is when  $L=L_c(j)$, that we need to take the parameter $\rho$ small
enough.

The method outlined above is rather accurate in finding the threshold for the wetting transition. In this respect, we remark that a direct ``energy vs.\ entropy'' argument such as the one used in \cite{CV} would
fail to capture the right dependence in the parameter $\si^2$ for instance.  Moreover, our method is robust enough to admit an extension to the setting of self-avoiding paths, as we discuss below.   

One of the motivations for this work stems from 
 the mathematical analysis of
contour models arising in 
low-temperature two-dimensional spin systems and related interface
models. In this context the random walk is
replaced by a self-avoiding and weakly self-interacting  random lattice path
with an effective diffusion constant $\s^2\sim e^{-\b}$, where $\b$ is
the inverse temperature. 
Here the analog of the pinning strength $\e_j$ above 
typically decays
like $e^{-\a \b(j +1) }$ for $j\to\infty$ with $\a > 1$. Whether such long range
potential is able to localize the contour is a key question in the
analysis of large deviations problems such as e.g.\ the Wulff
construction for the 2D Ising model \cite{DKS} and for the $(2+1)$-dimensional
Solid-on-Solid model \cite{CLMST}. We refer the interested
reader to \cite{IST} for more details.     

Our approach can be applied in principle to this setting.
In Section \ref{paths} below we work out  the details of this extension
in the simplified case where the
self-avoiding path has no additional  self-interaction.  
The general case has been recently solved in \cite{IST}, with stronger
results, with a
very different approach. The main
idea of
\cite{IST } consists in constructing, out of the self-interacting
contour path, an effective
random walk together with a renewal structure and then prove delocalization for the latter.

\section{Models and results}

\subsection{Random walks and pinning}
\label{RW}
We consider a class of symmetric and irreducible random walk kernels on $\bbZ$ with  variance $\si^2$. Since we are interested in the regime of small $\si^2$, we will make the assumption $\si^2\in(0, \tfrac12]$.
\begin{definition}\label{def:rw}
For a fixed constant $c_0\in(0,1]$ and $\si^2\in(0,\tfrac12]$, we call $\cP(c_0,\si^2)$ or simply $\cP(\si^2)$
the set of all symmetric probabilities $p:\bbZ\mapsto [0,1]$ with 
variance $\si^2$
such that 
\begin{equation}
p(1)\geq \tfrac12 \,c_0\,\si^2\quad \text{and}\quad \sum_{k\in\bbZ}|k|^3p(k)\leq c_0^{-1}\si^2.
\label{pps}
\end{equation}
 \end{definition}
In particular, any $p\in\cP(\si^2)$ satisfies $p(0)\geq 1-\si^2\geq
\tfrac12$, and the associated random walk is irreducible. Below we shall restrict ourselves to random walk kernels in the class $\cP(\si^2)$. 
While we do not believe this to be the largest possible class for our results to hold, the above assumptions turn out to be very convenient from the technical point of view.  At the same time, they 
include a wide range of interesting models.  
Two key examples to keep in mind are:
\begin{enumerate}[1)]
\item the symmetric nearest neighbor walk with $p(\pm1)=\tfrac{\si^2}{2}$, $p(0)=1-\si^2$ and $p(k)=0$ otherwise, referred to as the {\em binomial walk}, and 

\item the geometric walk with $p(k)=\tfrac1{Z_\b}e^{-\b |k|}$,
  $k\in\bbZ$, where $\b\in(0,\infty)$ is the unique positive solution of \begin{equation}\label{sigmabeta}
\si^2 = \frac{2e^{-\b }}{(1-e^{-\b })^2},
\end{equation} 
and $Z_\b=(1+e^{-\b })(1-e^{-\b })^{-1}$.  We refer to this as the {\em SOS
    walk} at inverse temperature $\b$, from its relation with the so-called Solid-On-Solid model. 
\end{enumerate}

\bigskip
We call $\bbP_i$ the law on trajectories of the random walk starting at $i\in\bbZ$.
Let $$\O_{0,L}= \{\g=(\g_0,\dots,\g_L):\; \g_i\in\bbZ,\, \g_0=\g_L=0\},$$ denote the set of trajectories
which start and end at zero, and define the partition function 
\begin{equation}\label{pfo}
Z_{0,L} = \sum_{\g\in\O_{0,L}} w(\g)\,,\qquad w(\g) = \prod_{i=1}^{L}p(\g_i - \g_{i-1}).
\end{equation}
We write $\bbP_{0,L}$ for the law of the walk conditioned to $\O_{0,L}$, that is for any $\g\in\O_{0,L}$:
\begin{equation}\label{pfo1}
\bbP_{0,L} (\g)=  \frac{\bbP_0(\g)}{\bbP_0(\O_{0,L})} = \frac{w(\g)}{Z_{0,L}}.
\end{equation}
For a fixed integer $j\geq 0$, consider the paths $\O_{0,L}^{+,j}$ that stay above height $-j$:  
\begin{align}\label{o+j}
\O_{0,L}^{+,j}=\{\g\in\O_{0,L}:\; \g_i\geq -j\; \,\forall i\}\,,\qquad Z^{+,j}_{0,L} = \sum_{\g\in\O_{0,L}} w(\g)\IND(\g\in\O^{+,j}_{0,L}).
\end{align}
The number of contacts with level zero is given by
\begin{align}\label{njo}
N(\g) = \sum_{i=1}^{L-1} \IND(\g_i=0).
\end{align}
For any $\e>0$ we consider the probability measures \begin{align}\label{po+j}
\bbP_{0,L}^{+,j}(\g) \propto w(\g) \IND(\g\in\O^{+,j}_{0,L}),\qquad \bbP_{0,L}^{\e,+,j}(\g) \propto e^{\e N(\g)} w(\g) \IND(\g\in\O^{+,j}_{0,L}),
\end{align}
and the corresponding expectations $\bbE_{0,L}^{+,j}$, $\bbE_{0,L}^{\e,+,j}$.  

\begin{maintheorem}\label{th1}
There exist constants $a\geq b>0$ and $c>0$, such that the following holds  
for any integer $j\geq 0$, any $\si^2\in(0,\tfrac12]$ and any random walk $p\in\cP(\si^2)$:
\begin{enumerate}[i)] 

\item If $\e\geq a \si^2 /(j+1)$, then  for all $L$ large enough
\begin{align}\label{pinwalla}
\bbE^{+,j}_{0,L}
\big[e^{\e N}\big] 
\geq \exp{\big(c\, \si^2 L\,(j+1)^{-2}\big)}.
\end{align}

\item If $\e\leq b  \si^2 /(j+1)$, then:
\begin{align}\label{pinwallb}
\bbE^{+,j}_{0,L}
\big[e^{\e N}\big] 
\leq 2 \frac{Z_{0,L}}{Z^{+,j}_{0,L}}.
\end{align}

\end{enumerate}
\end{maintheorem}

\begin{remark}\label{krem1}
It is well known that the ratio $Z_{0,L}/Z^{+,j}_{0,L}$ appearing in \eqref{pinwallb} is $O(L)$ as $L\to\infty$; see \eqref{posk1} below. In particular,  
the bounds in Theorem \ref{th1} imply that  a wetting transition occurs at a critical value $\e_c$
that satisfies $b\leq \e_c\si^{-2}(j+1)\leq a$, with constants $a,b$ that are independent of $\si^2$ and $j$ and independent of $p\in\cP(\si^2)$. This extends well known results in the case $j=0$; see e.g.\ \cite{IsoYos,Giacomin,DGZ,CGZ}.
\end{remark}

\begin{remark}\label{krem2}
In Proposition \ref{bootstrap} below we show that the upper bound \eqref{pinwallb} can be upgraded to the following bound independent of $L$, but with possibly non-optimal dependence on $j,\si^2$:
\begin{align}\label{pinwallbo}
\bbE^{+,j}_{0,L}
\big[e^{\e N}\big] 
\leq K,
\end{align}
for some constant $K=K(j,p)$, $p\in\cP(\si^2)$, whenever $\e\leq b  \si^2 /(j+1)$. 
\end{remark}

Next, we consider a more general interaction with the wall.  
For any $\g\in\O_{0,L}^{+,0}$, define the potential
\begin{align}\label{clpin1}
\Phi(\g) = \sum_{j=0}^\infty \e_j N_j(\g)\,,\qquad N_j(\g) = \sum_{i=1}^{L-1}\IND(\g_i=j),
\end{align}
 where $\{\e_j\}$ is a given nonnegative  sequence, and let
\begin{align}\label{zetaphi}
Z^{\Phi,+,0}_{0,L}=\sum_{\g\in\O^{+,0}_{0,L}}w(\g)e^{\Phi(\g)}.
\end{align}
 
 \begin{maintheorem}\label{th2}
There exist absolute constant $a,b,c>0$  such that, for any
$\si^2\in(0,\tfrac12]$ and any random walk $p\in\cP(\si^2)$,  the
following holds:
\begin{enumerate}[i)] 

\item For any integer $d\geq 0$ such that 
\begin{align}\label{bocpinwallb}
\frac1{d+1}\sum_{j=0}^{d/2} (j+1)^2\e_j\geq a \si^2, \end{align} 
we have
\begin{align}\label{cpinwallao}
 \bbE^{+,0}_{0,L}
\big[e^{\Phi(X)}\big]\geq 
\exp{\big(c\,\si^2 L(d+1)^{-2}\big) }, 
\end{align}
for all $L$ large enough.

\item
If the sequence $\{\e_j\}$ satisfies 
\begin{align}\label{ocpinwallb}
\sum_{j=0}^\infty (j+1)\e_j\leq b \si^2, \end{align} 
then 
\begin{align}\label{occpinwallb}
Z^{\Phi,+,0}_{0,L}
\leq 4\,Z_{0,L}.
\end{align}
\end{enumerate}

\end{maintheorem}

Notice that since $j\leq d$ in the summation \eqref{bocpinwallb}, the
condition for localization is slightly stronger than 
the bound $\sum_{j=0}^\infty (j+1)\e_j\geq a \si^2$ that would be
sharp in view point (ii). However, in many natural cases of interest,
condition \eqref{bocpinwallb} is actually rather sharp; see Corollary \ref{coro1} below.


\begin{remark}
\label{Jost} Our criterion \eqref{ocpinwallb} for the
  absence of a localized phase bears some similarity with the one
  derived in \cite{Jost} by Jost and Pais (cf. also the more general Bargmann's
  bounds in \cite{Bargman,Solomyak}) to exclude bound states
 for the Schr\"odinger equation 
\[
-\D\psi+V(r)\psi = E \psi
\]
in
  an attractive central potential $V\le 0$ in $\bbR^3$ or, after moving
  to radial coordinates, for the Sturm-Liouville problem on the
  half-line
\[
-\frac{d^2}{dx^2}f(x) +V(x)f(x)=\l f(x),\quad f(0)=0.
\] In \cite{Jost} it was proved in
  fact that
  if $\int_0^\infty dr \, r |V(r)| <1$ then there are no bound
  states. The connection between the pinning problem and the bound
  state problem goes as follows.  Let $\gamma(\cdot)$ be the random walk on
  $\bbZ$ with law $p\in \cP(\sigma^2)$ and let 
 $\bbE_0(\cdot)$ denote the average over the trajectories
  of $\gamma(\cdot)$ starting at the origin. If $\t$ denotes the hitting
  time of the half-line $(-\infty,-1]$ then we can write
\[
\frac{Z^{\Phi,+,0}_{0,L}}{Z^+_{0,L}}=
\bbE_0\left(e^{\sum_{s=0}^L F(\gamma(s))}\mid \gamma(L)=0;\, \t>L\right),
\]
where $F(x) = \sum_{j\geq 0} \e_j \IND(x=j)$.
If we pretend that the random walk 
behaves like a
Brownian motion with the correct diffusion constant and we replace $F$
with $V(x):=-F(\lfloor x\rfloor)$, then, using the Feynman-Kac
formula, we get that the r.h.s.\ above
has the form 
\begin{equation}
  \label{eq:settembre}
e^{-L \cH}(0,0)/p(0,0;L),
\end{equation}
where $\cH= -\s^2\frac{d^2}{dx^2} +V(x)$ acts on $\bbL^2((-1,+\infty))$ with
Dirichlet boundary conditions and $p(x,y;L)$ is the transition
probability density for the Brownian motion killed at $-1$. If the equation 
$\cH f=\l f$
has a solution bounded in $\bbL^2$ (of course with a negative eigenvalue $\l$) then the above
ratio should diverge exponentially fast in $L$. The
Jost-Pais criterium says that this cannot be the case if
$\int_{0}^{+\infty} dr\, r |V(r)| <\s^2$, which is indeed analogous to
our condition $\sum_{j\geq 0} \e_{j}(j+1)  < b\si^2$. Notice that the
absence of bound states does not guarantee that the ratio
\eqref{eq:settembre} stays bounded in $L$. In this sense our result is
stronger.  
\end{remark}

The above theorems allow us to identify rather precisely the critical point of the wetting transition for the generalized pinning problem described by \eqref{zetaphi} when the sequence $\e_j$ is given. For the sake of definiteness we mention only two types of sequences below, the power law and the exponential law. 
It is immediate to deduce the following corollary.
\begin{maincorollary}\label{coro1} 
Let the sequence $\e_j^0$ be either the power law $\e^0_j = (j+1)^{-2-\d}$ or the exponential law $\e^0_j=e^{-\d j}$, for some $\d>0$. Then   
there exist constants $a> b>0$ and $c>0$ such that the following holds
for any $\si^2\in(0,\tfrac12]$ and any $p\in\cP(\si^2)$:
\begin{enumerate}[i)] 

\item If the pinning sequence is $\e_j:= a\,\si^2 \e^0_j$, then the walk is localized: for $L$ large enough, 
\begin{align}\label{cpinw21}
 \bbE^{+,0}_{0,L}
\big[e^{\Phi(X)}\big]\geq 
\exp{\big(c\,\si^2 L\big) }; 
\end{align}

\item If the pinning sequence is $\e_j:= b\,\si^2 \e^0_j$, then the walk is delocalized:
\begin{align}\label{occpinw}
Z^{\Phi,+,0}_{0,L}
\leq 4\,Z_{0,L}.
\end{align}
\end{enumerate}
On the other hand, suppose that $\e^0_j =(j+1)^{-2+\d}$, for some
$\d>0$. Then, for any $b>0$, for any $\si^2>0$, any  walk $p\in\cP(\si^2)$
is localized by the pinning sequence $\e_j=b\,\si^2\e^0_j$.
 
\end{maincorollary}
  
\subsection{Self-avoiding lattice paths interacting with a wall}\label{latticepaths}
We turn to the description of the lattice path model.
We first define the class of lattice paths to be considered. 

\begin{definition}\label{contourdef}
We call $V:=(\bbZ +\tfrac12)\times \bbZ$ the vertex set of our lattice paths. 
An {\sl edge} is an unordered pair of points $e=\{x,y\}$, $x,y\in V$, with euclidean distance $d(x,y)=1$. 
An edge can be horizontal if $x = y \pm e_1$  or vertical if $x=y\pm e_2$, where $e_1=(1,0)$ and $e_2=(0,1)$.  
A {\sl self-avoiding lattice path} 
(for short a {\sl path} in the sequel) 
joining  $x\in V$ and $y\in V$ is a 
sequence $f_0,\ldots,f_n$ of edges such that:
\begin{enumerate}
\item for every $i=0,\dots,n-1$, $f_i$ and $f_{i+1}$ have one common vertex  $z_i\in V$;
\item $f_0=\{z_0,z_1\}$, $f_n=\{z_n,z_{n+1}\}$ with $z_0=x$ and $z_{n+1}=y$;
\item all vertices $z_i$, $i=0,\dots,n+1$, are distinct.
\end{enumerate}
We denote the length of a lattice path $\gamma$, that is the number of edges in $\g$, by $|\gamma|$.
Given $x,y\in V$, $x\neq y$, we call $\Omega(x,y)$ the set of all paths  joining $x$ and $y$.
\end{definition}

\vskip0.35cm
\begin{figure}[htb]
        \centering
 \begin{overpic}[scale=0.47]{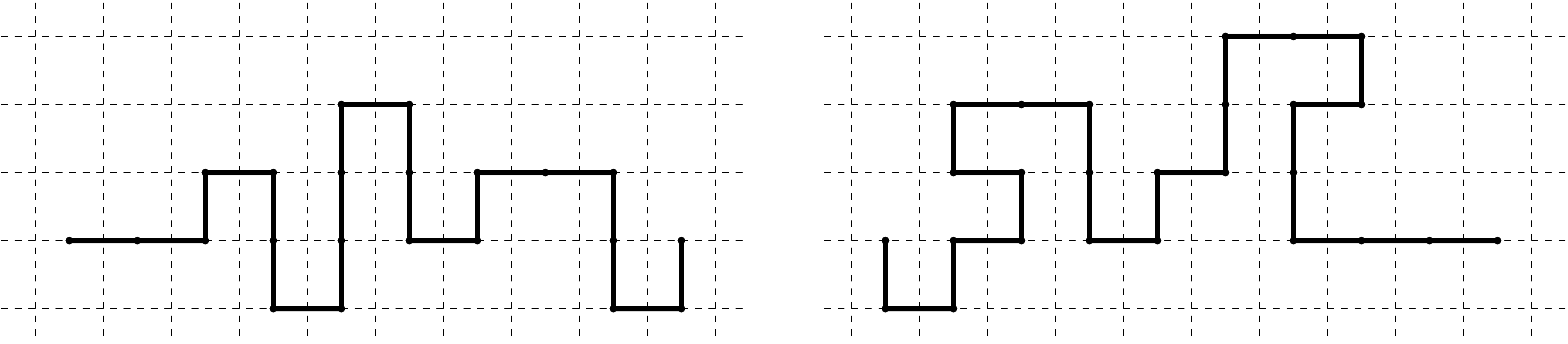}
\put(4,8) {$x$}\put(43,8) {$y$}
\put(56,8) {$x$}\put(95,8) {$y$}
\end{overpic}
        \caption{
        Two examples of lattice paths joining the vertices $x=(\tfrac12,0)$ and $y=(10-\tfrac12,0)$. In the background, the integer grid $\bbZ^2$.  
          }\label{fig:paths}
\end{figure}
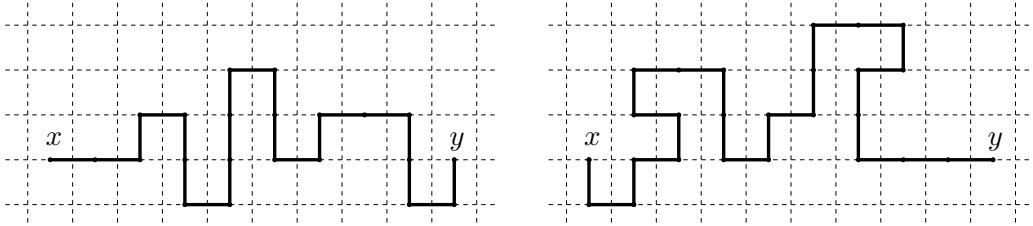
Notice that for $x=(x_1,x_2)$ and $y=(y_1,y_2)$, all  $\g\in \O(x,y)$ have at least $|y_1-x_1|$ horizontal edges. Paths that have exactly $ |y_1-x_1|$  horizontal edges are in one to one correspondence with random walk trajectories from $x-\tfrac12 e_1$ to $y+\tfrac12 e_1$; see Figure \ref{fig:paths}.

Next, we define the ensemble of lattice paths. 
Given $x,y\in V$, $x\neq y$, we 
write
\begin{align}\label{partfctso}
\cZ(x,y) = \sum_{\g\in\O(x,y)}e^{-\b|\g|},
\end{align}
where $\b>0$ is the inverse temperature parameter. With slight abuse of notation, when $x=(\tfrac12,0)$ and $y=(L-\tfrac12,0)$, we write $\O_{0,L}$ instead of $\O(x,y)$ and $\cZ_{0,L}$ instead of $\cZ(x,y)$. Observe that  
if we restrict 
to the set $\hat \O_{0,L}$ of paths with minimal number of horizontal edges then we obtain 
the following relation with the partition function $Z_{0,L} $ of the SOS walk with parameter $\b$:
\begin{align}\label{sosvsp}
\hat \cZ_{0,L} := \sum_{\g\in \hat \O_{0,L}} e^{-\b|\g|} = e^{-\b(L-1)}(Z_\b)^{L}Z_{0,L},
\end{align}
where $Z_\b = (1+e^{-\b})(1-e^{-\b})^{-1}$.

For any $\g\in\O_{0,L}$, define 
\begin{align}\label{clpin1p}
\Phi(\g) = \sum_{j=0}^\infty \e_j N_j(\g)\,,\qquad N_j(\g) = \sum_{i=1}^{L-2}\IND((i+\tfrac12,j)\in\g),
\end{align}
 where $\{\e_j\}$ is a nonnegative sequence. Define also 
\begin{align}\label{zetaphip}
\cZ^{\Phi,+,0}_{0,L}=\sum_{\g\in\O^{+,0}_{0,L}}e^{\Phi(\g)- \b|\g|},
\end{align}
where $\O^{+,0}_{0,L}$ denotes the set of paths $\g\in\O_{0,L}$ with $\g\geq 0$. With slight abuse of notation
we write again $\bbP_{0,L}$ and $\bbE_{0,L}$ (resp.\ $\bbP^{+,0}_{0,L}$ and $\bbE^{+,0}_{0,L}$) 
for the probability and expectation over paths $\g\in \O_{0,L}$ (resp.\ $\g\in\O^{+,0}_{0,L}$) with weight 
$e^{-\b|\g|}$.

 \begin{remark}[Alternative definitions of the number of contacts] \label{rem_zeros}
 In analogy with the random walk case one could have  considered the number $\hat N_j(\g)$ 
 of {\em horizontal} contacts of $\g$ with the $j$-th level rather than the number $N_j(\g)$ defined in \eqref{clpin1p}. By horizontal contact with the $j$-th level here we mean an horizontal edge $f\in\g$ at height $j$.  For instance the path in Figure \ref{fig:paths} (left) has $N_0(\g)=7$ and $\hat N_0(\g)=3$. 
 It is easy to check that all our results apply with no modifications to the potential $\hat \Phi$ obtained as in \eqref{clpin1p} with $N_j$ replaced by $\hat N_j$. As we will  see  in the proof of Theorem \ref{th2p}, our results actually extend to a model where one takes into account also possible contacts of $\g$ with the $j$-th level occurring outside of the horizontal interval $[0,L]$.
 \end{remark}

 \begin{maintheorem}[Delocalized phase]\label{th2p}
There exist universal constants $b>0$, $\b_0>0$  
such that the following holds for any $\b\geq\b_0$.
If $\sum_{j=0}^\infty (j+1)\e_j\leq b e^{-\b} $, then for all $L\in\bbN$
\begin{align}\label{cpinwallbp}
\cZ^{\Phi,+,0}_{0,L}
\leq 4\cZ_{0,L}.
\end{align}
\end{maintheorem}

\begin{maintheorem}[Localized phase]\label{th22p}
There exist universal constants $ a,c,\b_0>0$ 
such that the following holds:
For any sequence $\{\e_j\}$, any integer $d\geq 0$ and any $\b\geq \b_0$ such that 
\begin{align}\label{bocpp}
\frac1{d+1}\sum_{j=0}^d (j+1)^2\e_j\geq a e^{-\b}, \qquad e^{-\b}\leq c (d+1)^{-2},
\end{align}
one has 
\begin{align}\label{cpinw}
 \bbE^{+,0}_{0,L}
\big[e^{\Phi(X)}\big]\geq 
\exp{\big(c\,e^{-\b} L(d+1)^{-2}\big) }, 
\end{align}
for all $L$ large enough.
\end{maintheorem}

As in Corollary \ref{coro1} one can immediately infer from Theorem \ref{th2p} and Theorem \ref{th22p} the following facts.
\begin{maincorollary}\label{coro2}
Let the sequence $\e_j^0$ be either the power law $\e^0_j = (j+1)^{-2-\d}$ or the exponential law $\e^0_j=e^{-\d j}$, for some $\d>0$. Then   
there exist constants $a> b>0,c>0$ and $\b_0>0$ such that the following holds for any $\b\geq \b_0$:
\begin{enumerate}[i)] 

\item If the pinning sequence is $\e_j:= a\,e^{-\b} \e^0_j$, then the path is localized with 
\begin{align}\label{cpinw21p}
 \bbE^{+,0}_{0,L}
\big[e^{\Phi(X)}\big]\geq 
\exp{\big(c\,e^{-\b} L\big) }. 
\end{align}

\item If the pinning sequence is $\e_j:= b\,e^{-\b} \e^0_j$, then the path is delocalized with 
\begin{align}\label{occpinwp}
\cZ^{\Phi,+,0}_{0,L}
\leq 4\cZ_{0,L}.
\end{align}
\end{enumerate}
On the other hand, suppose that $\e^0_j = (j+1)^{-2+\d}$, for some $\d>0$. Then, for any $a>0$, the path is localized  
with the pinning sequence $ \e_j=ae^{-\b}\e^0_j$ for any $\b\geq \b_0$ for some constant $\b_0=\b_0(a,\d)>0$.
 \end{maincorollary}

\section{Random walks}
\label{pinwall}
Here we prove the main results for the random walk model. 

\subsection{Proof of the lower bounds}\label{lowborw}

We start with some considerations that apply to both the lower bound \eqref{pinwalla}
in Theorem \ref{th1} and to \eqref{cpinwallao}  in Theorem
\ref{th2}. First of all, we can assume that $\epsilon_j\le \log 2$ for each
$j$. Otherwise if $\epsilon_{j^*}>\log 2$, one gets immediately exponential growth of the
partition function using irreducibility of the walk and the fact that
$p(0)\geq 1/2$ (just consider the trajectory that reaches height $j^*$
and sticks there: its weight grows like $p(0)^Le^{\epsilon_{j^*}L}$).

Letting
\begin{align}
  \label{eq:14}
  W_L(i,j)=\sum_{\gamma\ge0, \gamma_0=i,\gamma_L=j}w(\gamma)
e^{\Phi(\gamma)}
\end{align}
one checks immediately that
\begin{align}
  \label{eq:15}
  W_{L+1}(i,j)=\sum_{z\ge0}W_L(i,z) e^{\epsilon_z} p(z-j).
\end{align}
Thus, viewing $W_L$ as a matrix and observing that $W_1=P$, with
$P$ the symmetric matrix $P_{ij}=p(i-j)1_{i,j\ge0}$, we have
\begin{align}
  \label{eq:16}
  W_L=e^{-V}(e^{V}P)^L,
\end{align}
with $V$ the
diagonal operator
$V(i,j)=1_{i=j}{\epsilon_j}$.
Let $\varphi=(\varphi_i)_{i\ge0}$ be a vector with $\varphi_i\ge0$,
set $\psi=e^{-V/2}\varphi$ and assume that 
$(\psi, \psi)=\sum_{i\ge0}\varphi_i^2e^{-\epsilon_i}=1$.
We first claim
\begin{align}
  \label{eq:17}
  (\varphi,W_L\varphi)\geq (\varphi, P\varphi)^L,
\end{align}
where $(\cdot,\cdot)$ is the scalar product in $\ell^2(\bbZ_+)$. 
To see this observe that, if $\tilde P$ is the self-adjoint operator $\tilde P=e^{V/2}P e^{V/2}$, then
\begin{align}
  \label{eq:18}
(\varphi,  e^{-V}(e^{V}P)^L\varphi)=(\psi,(\tilde P)^L\psi).
\end{align}
Since $p(0)\geq 1/2$, one has that $P$ is non-negative definite, and so is $\tilde P$. Letting $\mu_\psi(dE)$ denote the spectral measure of
$\tilde P$ associated with $\psi$ (the total mass is $1$ since $\psi$
has unit $\ell^2(\bbZ_+)$-norm), we have
\begin{align}
  \label{eq:19}
   (\psi,(\tilde P)^L\psi)=\int_{0}^\infty E^L \mu_\psi(dE)\geq 
\left(\int_{0}^\infty E \mu_\psi(dE)\right)^L= (\psi, \tilde P\psi)^L=(\varphi,P\varphi)^L
\end{align}
where we used convexity of $x^L$ on $[0,\infty)$, for $L\geq 1$, and Jensen's inequality.

Given an integer $d\geq 0$, let $s(i)=\sin(
\pi\frac{i+1}{d+2})
$ and
\begin{align}
  \label{eq:21}
  \varphi_i=\frac{s(i)}{\sqrt K}=\frac{s(i)}{\sqrt{\sum_0^{d}
      s(i)^2\exp(-\epsilon_i)}}\quad 0\le i\le d,
\end{align}
and $\varphi_i=0$ if $i<0$ or $i>d$.
We have then 
\begin{align}
  \label{eq:22}
  (\varphi,P\varphi)=(\varphi,\varphi)-(\varphi,(1-P)\varphi)=
\frac1K\left[\sum_0^{d}s(i)^2-\frac12\sum_{i,j\le d}P_{ij}(s(i)-s(j))^2\right]
\end{align}
where we used the symmetry of $P$. Since 
\begin{align}
  \label{eq:23}
|s(i)-s(j)|\le \pi\frac{|i-j|}{d+1} 
\end{align}
we have 
\begin{align}
  \label{eq:24}
  \frac12\sum_{i,j\le d}P_{ij}(s(i)-s(j))^2\le \frac{\pi^2}2\frac{\sigma^2}{(d+1)}.
\end{align}
On the other hand, recalling that  $0\le \epsilon_j\le \log
2$ and using $\exp(-x)\le
1-x/4$ for $0\le x\le \log2$,
\begin{align}
  \label{eq:25}
  K\le \sum_0^{d/2}s(i)^2(1-\epsilon_i/4)+\sum_{d/2+1}^{d}s(i)^2.
\end{align}
Finally, since $s(i)\geq c (i+1)/(d+1)$ for $0\le i\le d/2$ for some
positive constant $c$, we conclude that
\begin{align}
  \label{eq:26}
  (\varphi,P\varphi)\ge\frac{\sum_0^{d}s(i)^2-\frac{\pi^2}2\frac{\sigma^2}{(d+1)}}{\sum_0^{d}s(i)^2-\frac
    c{4(d+1)^2}\sum_0^{d/2}
  (i+1)^2\epsilon_i}.
\end{align}

\begin{proof}[Proof of \eqref{cpinwallao}]
  Assume by monotonicity that we have equality in
  \eqref{bocpinwallb}. Then we go back to \eqref{eq:26}, we observe
  that $\sum_0^{d}s(i)^2$ grows linearly in $d$ and that we can assume
  that $d$ is much larger than $\sigma^2/d$. Then we obtain, for some
  universal positive constant $C$,
  \begin{eqnarray}
    \label{eq:27}
     (\varphi,P\varphi)\geq 1-C\frac{\sigma^2}{(d+1)^2}+\frac
     aC\frac{\sigma^2}{(d+1)^2}\geq 
1+\frac
     a{2C}\frac{\sigma^2}{(d+1)^2}
  \end{eqnarray}
if the value of $a$ in Theorem \ref{th2} is chosen large enough.
Recalling \eqref{eq:17} we see that there exist $i,j\le d$ and 
positive constants $c(d),a'$ such that
\begin{align}
  \label{eq:29}
W_L(i,j)\geq c(d)\exp\left[L a'\frac{\sigma^2}{(d+1)^2}\right].  
\end{align}
In turn, using the assumption $p(1)>0$, we obtain
\begin{align}
  \label{eq:30}
  W_L(0,0)\geq p(1)^{i+j}c'(d) \exp\left[L a'\frac{\sigma^2}{(d+1)^2}\right],
\end{align}
for  some new constant $c'(d)>0$. This implies the desired lower bound \eqref{cpinwallao}.
\end{proof}

\begin{proof}[Proof of \eqref{pinwalla}]

By vertical translation invariance, we can assume that the walk is conditioned
to stay non-negative (instead of $\gamma \geq -j$) and that the pinning is at height $j$,
i.e. $\varepsilon_i=\epsilon 1_{i=j}$.
The proof of \eqref{pinwalla} is then essentially identical to the proof of \eqref
{cpinwallao}, once the choice
 $d=2j$ is made in the definition of $s(\cdot)$ above. 
\end{proof}

\subsection{Proof of the upper bound in Theorem \ref{th1}}\label{proofubj}
Here we prove \eqref{pinwallb}. Define 
\begin{align}\label{oe+j}
Z^{\e,+,j}_{0,L} = \sum_{\g\in\O_{0,L}} e^{\e N(\g)}w(\g)\IND(\g\in\O^{+,j}_{0,L}).
\end{align}
Thus, \eqref{pinwallb} can be restated as follows
\begin{proposition}\label{propoz}
There exists a universal constant $b>0$ such that if $\e\leq b\si^2/(j+1)$, then uniformly in $L,\si^2,j$ one has 
\begin{align}\label{propoz1}
Z^{\e,+,j}_{0,L} \leq 2 Z_{0,L}. 
\end{align}
\end{proposition}
The proof of Proposition \ref{propoz} is divided in two steps. We start with a general lemma for the walk with no wall constraint. This will allow us to cover the region $L\leq C  (j+1)^2/\si^{2}$ for any constant $C>0$. 
\begin{lemma}\label{small}
There exists $c>0$ such that for any $L$ and $\si^2\in(0,\tfrac12]$, any $p\in\cP(\si^2)$, 
for any $b<1/c$ one has
\begin{align}\label{check2}
\bbE_{0,L}
\big[e^{b \si N/\sqrt {L}}\big]\leq 1+c\,b.
\end{align}
\end{lemma}
\begin{proof}
Writing $N$ as in \eqref{njo}, for $\g\in\O_{0,L}$ one has the expansion
\begin{align}\label{expando}
e^{b \si N(\g)/\sqrt {L}} & = \prod_{i=1}^{L-1}e^{b\si \IND(\g_i=0)/\sqrt {L}}\\
&\leq \prod_{i=1}^{L-1}(1+2b\si \IND(\g_i=0)/\sqrt {L})\\& = 1+\sum_{n=1}^{L-1}(2b\si L^{-1/2})^n
\!\!\!\sum_{x_1<\cdots<x_n}\IND(\g_{x_1}=\dots\g_{x_n}=0),
\end{align}
where we use $e^a\leq 1+ 2a$ for $a\in (0,1)$, and the sum ranges over all possible positions of the internal zeros. We also set $x_0=0, x_{n+1}=L$. We can assume that $\si^2 L$ is large, since otherwise the statement becomes obvious by estimating $N\leq L$. 
Observe that
\begin{equation}\label{alocalzeros}
\bbP_{0,L}(\g_{x_1}=\dots\g_{x_n}=0)=\frac{Z_{0,x_1}\cdots Z_{x_{n},L}}{Z_{0,L}},
\end{equation}
where we write $Z_{u,v}= Z_{0,v-u}$. From 
Proposition \ref{prop:2}  in the appendix  one has the local CLT estimates
\begin{equation}\label{alocalclt}
Z_{x_i,x_{i+1}}\leq C\,\si^{-1}(x_{i+1}-x_{i})^{-1/2},
\end{equation}
and $Z_{0,L}\geq C^{-1}\si^{-1}L^{-1/2}$, for some absolute constant $C>0$. 
It follows that 
\begin{align}\label{localzeros}
\bbP_{0,L}(\g_{x_1}=\dots\g_{x_n}=0)&\leq C^n\,\si\sqrt L\prod_{i=1}^{n+1}\si^{-1}(x_{i}-x_{i-1})^{-1/2}
\\& = C^n\si^{-n} \sqrt L\prod_{i=1}^{n+1}(x_{i}-x_{i-1})^{-1/2},
\end{align}
for some absolute constant $C>0$.
Therefore, 
\begin{align}\label{expando1}
\bbE_{0,L}\big[e^{b \si N/\sqrt {L}}\big]
 & \leq  1 + 
\sum_{n=1}^{L-1}(2bC)^nL^{-\frac{n-1}2}\!\!\!
\sum_{x_1<\cdots<x_n}\prod_{i=1}^{n+1}(x_{i}-x_{i-1})^{-1/2}.
\end{align}
It remains to check that for any $n\geq 1$, for some new absolute constant $C>0$ one has: 
 \begin{align}\label{expando2}
\sum_{x_1<\cdots<x_n}\prod_{i=1}^{n+1}(x_{i}-x_{i-1})^{-1/2}\leq C^n L^{\frac{n-1}2}.
\end{align}
Once \eqref{expando2} is available, it is immediate to conclude that
\eqref{check2} holds if $b$ is small.

To prove \eqref{expando2} 
we change variables to $\xi_i=x_{i}-x_{i-1}\in\bbN$, $i=1,\dots,n+1$. Thus 
 \begin{align}\label{expando3}
\sum_{x_1<\cdots<x_n}\prod_{i=1}^{n+1}(x_{i}-x_{i-1})^{-1/2} = 
\sum_{\xi_1,\dots,\xi_{n+1}:\,\sum_{j=1}^{n+1}\xi_j=L} (\xi_1\cdots \xi_{n+1})^{-1/2}. 
\end{align}
One easily checks that there is $C>0$ independent of $n,L$ such that for any fixed values of $\xi_1,\dots,\xi_{n-1}$ one has
$$
\sum_{\xi_{n},\xi_{n+1}:\,\xi_n+\xi_{n+1}=L-\sum_{i=1}^{n-1}\xi_i} (\xi_n \xi_{n+1})^{-1/2}\leq C.
$$ 
Therefore the sum in \eqref{expando3} can be bounded by
 \begin{align}\label{expando03}
C\prod_{j=1}^{n-1}\sum_{1\leq \xi_j\leq L} \xi_j^{-1/2}\leq (C')^nL^{\frac{n-1}2},
\end{align}
for some new constant $C'>0$.
%
\end{proof}

Lemma \ref{small} implies that for any small constant $b$, 
taking $L\leq (j+1)^2/\si^2$, and $\e\leq b\si^2/(j+1)$, then 
\begin{align}\label{smao}
Z^{\e,+,j}_{0,L}
\leq Z_{0,L}\bbE_{0,L}
\big[e^{\e N}\big] \leq (1+c\,b) Z_{0,L},
 \end{align}
where the first inequality follows by dropping the wall constraint $\g\geq -j$, while the second one is implied by \eqref{check2}.  

\begin{definition}\label{defhn}
Fix $L_0 = \frac C2(j+1)^2/\si^2$ (with $C$ the same constant as in
Proposition \ref{prop:2})  and set $L_n: = 2^n L_0$. For $n\in\bbN$, $\d>0$, and $b>0$, let $\cH_n(\d,b)$ denote the following statement: 
\[
Z^{\e,+,j}_{0,L}\leq (1+\d)  Z_{0,L} \quad \forall L\in [1,L_n],
\]
where $\e=b\si^2/(j+1)$.
\end{definition}
From \eqref{smao} we know that for any $\d>0$, there is some $b_0(\d)>0$ such  that $\cH_1(\d,b)$ holds for all $b\leq b_0(\d)$. The proof of Proposition \ref{propoz} is then completed by the following induction.
 
\begin{lemma}\label{indu}
There exist $\d\in(0,1)$, and $b\in(0,b_0(\d))$ such that,  for all $n\geq 1$, $\cH_n(\d,b)$ implies $\cH_{n+1}(\d,b)$.
\end{lemma}
\begin{proof} Fix $\d>0$ and $L\in(L_{n},L_{n+1}]$, and assume the validity of $\cH_n(\d,b)$.  
Define $\xi$ as the last zero of the walk up to $L/2$, and $\eta$ as the first 
zero of the walk beyond $L/2$:
\begin{align}\label{xieta}
\xi(\g)= \max\{i\leq L/2: \;\g_i=0\}\,,\quad \eta(\g)= \min\{i> L/2: \;\g_i=0\}.
\end{align} 
Then
\begin{align}\label{indu1}
Z^{\e,+,j}_{0,L} &= \sum_{x\leq L/2}\sum_{y>L/2}\sum_{\g\in\O_{0,L}} 
e^{\e N(\g)}w(\g)\IND(\g\in\O^{+,j}_{0,L})\IND(\xi(\g)=x,\eta(\g)=y)
\nonumber\\
& \leq  \sum_{x\leq L/2}\sum_{y>L/2}
e^{2\e} Z_{0,x}^{\e,+,j}\Xi_{x,y}Z_{y,L}^{\e,+,j},
\end{align}
where we use the notation $Z_{y,L}^{\e,+,j} = Z_{0,L-y}^{\e,+,j}$ with the convention that $Z_{0,0}^{\e,+,j}=Z_{L,L}^{\e,+,j}=1$ and we define 
\begin{align}\label{indu10}
\Xi_{x,y}:=\sum_{\g:\,x\to y}\IND(N(\g)=0,\g\geq -j)\prod_{i=x+1}^{y}p(\g_{i}-\g_{i-1}).
\end{align}
Here $\g:\,x\to y$ means that $\g = \{\g_i\,,\;i=x,\dots,y\}$ is a path such that $\g_x=\g_y=0$. 
Since $x\in[0,L/2]$,  $L-y\in[0,L/2)$, one has $Z_{0,x}^{\e,+,j}\leq (1+\d) Z_{0,x}$ and $Z_{y,L}^{\e,+,j}\leq (1+\d)Z_{y,L}$. Therefore \eqref{indu1} implies
\begin{align}\label{indu2}
Z^{\e,+,j}_{0,L} \leq  (1+\d)^2 e^{2\e}\!\!\!\sum_{x\leq L/2}\sum_{y>L/2}
 Z_{0,x}\Xi_{x,y}Z_{y,L}.
\end{align}
Let $E$ denote the event that $\g_{\lfloor L/2\rfloor} \geq - j$ and $\g_{\lfloor L/2\rfloor +1} \geq - j$. Clearly,
\begin{align}\label{indu3}
\frac1{Z_{0,L}}\sum_{x\leq L/2}\sum_{y>L/2}
 Z_{0,x}\Xi_{x,y}Z_{y,L}\leq \bbP_{0,L}(E).
\end{align}
From the estimate (b) in Proposition  \ref{prop:2} in the appendix, one has that $\bbP_{0,L}(E)\leq 3/4$,
uniformly in $L\geq L_1=C(j+1)^2/\si^2$, $j\geq 0$ and $\si^2>0$. Thus, \eqref{indu2} and \eqref{indu3}
imply 
\begin{align}\label{indu4}
Z^{\e,+,j}_{0,L} \leq  \tfrac34(1+\d)^2 e^{2\e} Z_{0,L}.
\end{align}
Since $\e=b\si^2/(j+1)\leq b$, if $\d$ and $b$ are small enough one has $\tfrac34(1+\d)^2 e^{2\e}\leq 1+\d$. This implies $\cH_{n+1}(\d,b)$. \end{proof}

Next, we turn to the proof of the upper bound \eqref{pinwallbo} announced in Remark \ref{krem2}.
It can be 
restated as follows.
\begin{proposition}\label{bootstrap}
Taking $\e\leq b\si^2/(j+1)$ as in Proposition \ref{propoz} one has, for all $L,j$, for all random walk kernel $p\in\cP(\si^2)$:
\begin{align}\label{posk}
Z^{\e,+,j}_{0,L}
\leq K \,Z^{+,j}_{0,L},
 \end{align}
 where $K= K(j,p)$  is independent of $L$.
\end{proposition}
\begin{proof}
We first recall the well known bounds  (see e.g.\ \cite{Doney}):
\begin{align}\label{posk1}
c_1\,\frac{1}{L}\leq \frac{Z^{+,j}_{0,L}}{Z_{0,L}} =\bbP_{0,L}(\g \geq -j)\leq c_2\,\frac{1}{L},
 \end{align}
for some constants $c_i=c_i(j,p)>0$ independent of $L$. 

From Proposition \ref{propoz} one has $Z^{\e,+,j}_{0,L}
\leq 2Z_{0,L}$ for all $L$. 
By H\"older's inequality, for $k\in\bbN$ and $\e'= \e/k$,
\[
\frac{Z^{\e',+,j}_{0,L}}{Z_{0,L}}=\bbE_{0,L}\left[\IND(\g \geq -j)e^{\e'
  N(\g)}\right]\le 2^{\frac1k} \bbP_{0,L}(\g \geq -j)^{\frac{k-1}{k}}. 
\] 
If $k= 3$, using \eqref{posk1} to estimate $\bbP_{0,L}(\g \geq -j)$ from above, 
we obtain 
\begin{align}\label{posk2}
\frac{Z^{\e/3,+,j}_{0,L}}{Z_{0,L}}\leq \k L^{-2/3},
\end{align}
with $\k = \k(j,p)>0$. 
We claim that we can bootstrap the bound \eqref{posk2} to 
\begin{align}\label{posk3}
\frac{Z^{\e/3,+,j}_{0,L}}{Z_{0,L}}\leq \k' L^{-1},
\end{align}
for some new constant $\k'=\k'(j,p)$.
Once this is achieved, the proposition follows by using the left side in \eqref{posk1}.

To prove \eqref{posk3}, we replace $\e/3$ by $\e$ for ease of notation. Consider  the decomposition \eqref{indu1}.  The bound \eqref{posk2} implies 
$$
Z_{0,x}^{\e,+,j}Z_{y,L}^{\e,+,j} \leq \k^2 (x+1)^{-2/3}(L-y+1)^{-2/3}Z_{0,x}Z_{y,L}.
$$
Moreover, neglecting the constraint $\g\geq -j$ in \eqref{indu10} one has that  $\Xi_{x,y}\leq 2Z_{x,y}^{+,0}$. Therefore by \eqref{posk1}, now with $j=0$, one has
$$\Xi_{x,y}\leq \k_0 (y-x)^{-1}Z_{x,y},$$
for some $ \k_0=\k_0(p)>0$.
Summarizing, we obtain
\begin{gather*}
 Z^{\e,+,j}_{0,L}\leq \k^2\k_0 \sum_{x\le L/2}\sum_{y> L/2}Z_{0,x}Z_{x,y}Z_{y,L} \ (x+1)^{-2/3}(L-y+1)^{-2/3}(y-x)^{-1}.
\end{gather*}
From the local CLT estimates in \eqref{alocalclt} one has 
\[
\frac{Z_{0,x}Z_{x,y}Z_{y,L}}{Z_{0,L}}= \bbP_{0,L}(\g_x=\g_y=0)\leq \k_1
\frac{1}{\sqrt{x}}\max(\frac{1}{\sqrt{y-x}}, \frac{1}{\sqrt{L-y}}),
\]
for some $ \k_1=\k_1(p)>0$.
In conclusion, for $\k_2:=\k^2\k_0\k_1$,
\[
\frac{Z^{\e,+,j}_{0,L}}{Z_{0,L}}\leq \k_2 \sum_{x\le L/2}\sum_{y>
  L/2}\frac{1}{x^{2/3}(L-y)^{2/3}(y-x)}
\frac{1}{\sqrt{x}}\max(\frac{1}{\sqrt{y-x}}, \frac{1}{\sqrt{L-y}}).
\]
We need to show that the r.h.s.\ above is $O(1/L)$. To this end, it suffices to consider the two sums
\begin{gather*}
A:= \sum_{x\le L/2}\sum_{y>
  L/2}\frac{1}{x^{2/3}(L-y)^{2/3}(y-x)\sqrt{x}\sqrt{y-x}},\\
B:= \sum_{x\le L/2}\sum_{y>
  L/2}\frac{1}{x^{2/3}(L-y)^{2/3}(y-x)\sqrt{x}\sqrt{L-y}}.
\end{gather*}
It is not hard to see that the sum in $A$ is always $o(L^{-1})$.
On the other hand, the sum in $B$ behaves as  $1/L$ as one easily sees by restricting to 
either $x\in [1,L/4]$ or $y\in [3L/4,L]$. 
\end{proof}

\subsection{Proof of Theorem \ref{th2}}\label{proofth2}
Let $\r_j:= (b\si^2)^{-1}(j+1)\e_j$. By monotonicity we may assume that $\sum_{j=0}^\infty\r_j=1$. Then, using Jensen's inequality
\begin{align}\label{bot0}
Z^{\Phi,+,0}_{0,L}&\leq\sum_{\g\in\O^{+,0}_{0,L}}w(\g)\sum_{j=0}^\infty\r_j e^{\k_j N_j(\g)}
\nonumber \\ &= \sum_{j=0}^\infty\r_j
Z^{+,0}_{0,L}\,\bbE^{+,0}_{0,L}\left[
e^{\k_jN_j}\right],
\end{align}
where $\k_j:=\e_j/\r_j = b\si^2/(j+1)$.

Next, we show that the upper bound in Proposition \ref{propoz} implies
\begin{align}\label{bot1}
\bbE^{+,0}_{0,L}\left[
e^{\k_jN_j}\right]\leq 4\,\frac{Z_{0,L}}{Z^{+,0}_{0,L}},
\end{align} 
if $b$ is small enough, uniformly in $j,\si^2$. 
Notice that this and  \eqref{bot0} imply the desired estimate \eqref{occpinwallb}.

To prove \eqref{bot1}, observe that by a vertical  translation, Proposition \ref{propoz} refers to the case where the walk
starts and ends at level $j$, with a wall at zero. 
Thus, writing $ Z_{u,v}^{\k_j,+,j}$ for  the partition function of  the walk with pinning strength $\k_j$ at level $j$, that starts at $j$ at time $u$ and ends at $j$ at time $v$, with wall at zero, 
Proposition \ref{propoz}  yields the bound
\begin{align}\label{bot2}
Z^{\k_j,+,j}_{u,v}
\leq 2\,Z_{u,v},
\end{align}
for any $0\leq u\leq v\leq L$.
Therefore \eqref{bot1} follows by writing 
\begin{align}\label{bot012}
Z^{+,0}_{0,L}\bbE^{+,0}_{0,L}\left[
e^{\k_jN_j}\right] = \sum_{\g_1}w(\g_1) +  \sum_{\g_2}w(\g_2)e^{ \k_jN_j(\g_2)},
\end{align}
where the first sum is over all paths $\g_1\geq 0$ joining the vertices $(0,0)$ and $(L,0)$ that never reach level $j$, while the second sum is over
all paths $\g_2\geq 0$ joining the same vertices  and that touch level $j$ at least once. The first sum is trivially bounded by $Z_{0,L}$. The second sum is bounded by summing over the first and last contact with the level $j$, so that 
\begin{align}\label{ubot}
\sum_{\g_2}w(\g_2)e^{ \k_jN_j(\g_2)}\leq \sum_{u\leq v}e^{2\k_j}\hat Z^{+,0}_{0,u}Z_{u,v}^{\k_j,+,j}\hat Z^{+,0}_{v,L},
\end{align}
where $\hat Z^{+,0}_{0,u}$ stands for the sum of $w(\g)$ over all $\g\geq 0$
joining $(0,0)$ and $(u,j)$ that  never touch height $j$ before the ending point. Similarly, $\hat Z^{+,0}_{v,L}$ stands for the sum of $w(\g)$ over all $\g\geq 0$
joining $(v,j)$ and $(L,0)$ that  never touch height $j$ after the starting point. 
Using \eqref{bot2} one obtains immediately that \eqref{ubot}
is bounded by $2e^{2 \k_j}Z_{0,L}$. Since $\k_j\leq b$ it follows that \eqref{bot1}
holds as soon as $1+2e^{2b}\leq 4$. 

\section{Self-avoiding paths}
\label{paths}
In this section we prove Theorem \ref{th2p} and Theorem \ref{th22p}. The strategy follows closely the corresponding arguments in the random walk case.  

\subsection{Lower bound}\label{lowbop}
\begin{proof}[Proof of Theorem \ref{th22p}]
Recall that if we restrict to the set $\hat\O_{0,L}$ of paths with minimal number of horizontal edges then we obtain the SOS walk with parameter $\b$; see \eqref{sosvsp}. We write $Z_{0,L}, Z^{+,0}_{0,L}$ for the usual partition functions of the SOS walk with weight $w(\g)$; see \eqref{pfo} and \eqref{o+j}.
By restricting to $\hat\O_{0,L}$ we may write 
\begin{align}\label{pcheckp}
\bbE^{+,0}_{0,L}
\big[e^{\Phi}\big] = \frac{\cZ^{\Phi,+,0}_{0,L}}{\cZ^{+,0}_{0,L}} &\geq 
\frac{e^{-\b(L-1)}(Z_\beta)^L}{\cZ^{+,0}_{0,L}}\sum_{\g\in\hat\O_{0,L}}w(\g)e^{\Phi(\g)} 
\nonumber\\
& \geq \frac{Z^{+,0}_{0,L}}{Z_{0,L}}\frac{Z_{0,L}}{\cZ_{0,L}}\,e^{-\b(L-1)}(Z_\b)^{L} \,\bbE_{0,L}^{+,0,{\rm SOS}}[e^{\Phi}],
\end{align} 
where the second line follows from the obvious bound $\cZ^{+,0}_{0,L}\leq \cZ_{0,L}$ and the expectation $\bbE_{0,L}^{+,0,{\rm SOS}}$ refers now to the SOS walk.
From \cite[Eq.\ (4.8.4)]{DKS}
one has that $$\hat \cZ_{0,L} \geq \cZ_{0,L}\exp{(-c e^{-2\b} L)},$$ for some absolute constant $c>0 $ and $L$ large enough, where $\hat \cZ_{0,L}$ is defined in \eqref{sosvsp}. Therefore,
\begin{align}\label{pcheckp1}
\frac{Z_{0,L}}{\cZ_{0,L}}\,e^{-\b(L-1)}(Z_\b)^{L} =
\frac{\hat{\mathcal Z}_{0,L}}{\mathcal Z_{0,L}}\geq \exp{(-c e^{-2\b} L)},
\end{align} for some absolute constant $c>0 $ and $L$ large enough. Moreover, by \eqref{posk}
one has $Z^{+,0}_{0,L}\geq c(\b)L^{-1}Z_{0,L}$ for some constant $c(\b)>0$.
Finally, the conclusion follows from estimating from below $ \bbE_{0,L}^{+,0,{\rm SOS}}[e^{\Phi}]$ as in \eqref{cpinwallao}.

\end{proof}

\subsection{Upper bound}\label{upbop}
We consider the probability measure $\bbP_{0,L}$ on $\O_{0,L}$ defined by the weight $e^{-\b|\g|}$, and write $\bbE_{0,L}$ for the expectation w.r.t.\ $\bbP_{0,L}$. 
%
%
Call $N(\g)=N_0(\g)$ the number of contacts with the zero line, as in the definifion \eqref{clpin1p}. 
For lightness of notation, below  we set $$\si^2:=e^{-\b}.$$ We have the following version of Lemma \ref{small}.
\begin{lemma}\label{smallscalep}
There exists $c>0,\b_0>0$ such that for any $L$ and $\b\geq \b_0$, for any $0<b<1/c$ one has
\begin{align}\label{check20p}
\bbE_{0,L}
\big[e^{b\, \si N(\g)/\sqrt {L}}\big]\leq 1+c\,b.
\end{align}
\end{lemma}
\begin{proof}
We proceed as in Lemma \ref{small}. Define $\chi_i (\g)= \IND((i+\tfrac12,0)\in\g)$. 
We have
\begin{align}\label{expandop}
e^{b \si N(\g)/\sqrt {L}} & = \prod_{i=1}^{L-1}e^{b\si \chi_i(\g)/\sqrt {L}}\\& \leq 1+\sum_{n=1}^{L-1}(2b\si L^{-1/2})^n
\!\!\!\sum_{x_1<\cdots<x_n}\chi_{x_1}\cdots\chi_{x_n},
\end{align}
where the sum ranges over all values of the integers $1\leq x_1<\cdots<x_n\leq L-2$. 
Let us prove that 
\begin{align}\label{localzerosp2p}
\bbE_{0,L}(\chi_{x_1}\cdots\chi_{x_n})
&\leq C^n\,\si^{-n}\sqrt L\prod_{i=1}^{n+1}(x_{i}-x_{i-1})^{-1/2},
\end{align}
where $C$ is a universal constant. 
Once \eqref{localzerosp2p} is available, the rest of the proof is exactly as in Lemma \ref{small}.

Given a permutation $\pi=(\pi(1),\dots,\pi(n))$, call $E_\pi$ the event that in going from $(\tfrac12,0)$ to $(L-\tfrac12,0)$, the path $\g$ visits the zeros at $\{x_1+\tfrac12,\dots,x_n+\tfrac12\}$ in the order $(x_{\pi(1)}+\tfrac12,\dots,x_{\pi(n)}+\tfrac12)$. One has
\begin{align}\label{localzerosp2}
\bbE_{0,L}(\chi_{x_1}\cdots\chi_{x_n};E_\pi)&\leq \frac{\cZ_{0,x_{\pi(1)}+1}\cZ_{x_{\pi(1)},x_{\pi(2)}+1}\cdots\cZ_{x_{\pi(n)},L}}{\cZ_{0,L}}.
\end{align}
Using Proposition \ref{prop:3} in the appendix it follows that 
\begin{align}\label{localzerosp22p}
\frac{\cZ_{0,x_{\pi(1)}+1}\cZ_{x_{\pi(1)},x_{\pi(2)}+1}\cdots\cZ_{x_{\pi(n)},L}}{\cZ_{0,L}}
&\leq \frac{C^n\,\si^{-n}\sqrt L}{\Xi_L}\,\prod_{i=1}^{n+1}
\frac{\Xi_{|x_{\pi(i)}-x_{\pi(i-1)}|}}{|x_{\pi(i)}-x_{\pi(i-1)}|^{1/2}},
\end{align}
where $\Xi_k$ denotes the grand-canonical partition function (the one
where the horizontal coordinate of the endpoint of the path is $L$
while the vertical coordinate is free), and we set $x_{\pi(0)}=0$ and $x_{\pi(n+1)}=L$. 
Since for any permutation $\pi$ one has
$$
\prod_{i=1}^{n+1}
(x_{\pi(i)}-x_{\pi(i-1)})^{-1/2}\leq \prod_{i=1}^{n+1}
(x_{i}-x_{i-1})^{-1/2},
$$
to establish \eqref{localzerosp2p} it is sufficient to prove that for some constant $C>0$ :
\begin{align}\label{askapa1}
\sum_{\pi}\frac{1}{\Xi_L}\,\prod_{i=1}^{n+1}
\Xi_{x_{\pi(i)}-x_{\pi(i-1)}}\leq C^n
\end{align}
As in \cite[Eq.\ (4.8.6)]{DKS}, one has that 
\begin{align}\label{askapa02}
\Xi_{L}\leq C_1 e^{-c(\b)L}, \;\;\; \Xi_L\geq C_2 e^{-c(\b)L},
\end{align}
for some absolute constants $C_1,C_2>0$, where $c(\b)>0$ is a constant such that $c(\b)\to\infty$ as $\b\to\infty$. 
It follows that for any $\pi$:
\begin{align}\label{askapa2}
\frac{1}{\Xi_L}\,\prod_{i=1}^{n+1}
\Xi_{x_{\pi(i)}-x_{\pi(i-1)}}\leq C^{\,n} e^{-c(\b) \ell(\pi)}, 
\end{align}
where we 
define the excess length associated to a permutation $\pi$ by
$$
\ell(\pi):=-L+\sum_{i=0}^n|x_{\pi(i+1)}-x_{\pi(i)}|.
$$
Notice that $\ell(\pi)\geq 0$ and $\ell(\pi)=0$ iff $\pi$ is the identity $\pi(i)\equiv i$. 
To conclude, we show that 
\begin{align}\label{asapa1}
\sum_{\pi}e^{-c(\b) \ell(\pi)}\leq (1+\d(\b))^n,
\end{align}
where $\d(\b)\to 0$ when $\b\to\infty$.   
The statement \eqref{asapa1} can be obtained by induction over $n$, as follows. 
Fix $x_0=0$ and $x_{n+1}=L$ and let $x_1,\dots,x_n$ be arbitrary integers satisfying $x_i\neq 0$ and $x_{1}<\cdots< x_{n}<L$. Here we do not assume that $x_1> 0$. 
Let $\phi(n)$ be defined as the sum in \eqref{asapa1} for this choice of points $\{x_i\}$. We claim that 
\begin{align}\label{asapa2}
\phi(n)\leq (1+\d(\b))^ne^{-c(\b)x_1^-},
\end{align}
where $x_1^-:= \max(-x_1,0)$ is the negative part of $x_1$. Clearly, \eqref{asapa2} is sufficient to prove \eqref{asapa1}
which corresponds to the case $x_1>0$. To prove \eqref{asapa2} notice that for $n=1$ one has $\pi(1)=1$ and $\ell(\pi)= 2x_1^-$ so that  \eqref{asapa2} is satisfied in this case. 
For $k\in\bbN$ and $\d>0$, assume $\phi(k)\leq (1+\d)^ke^{-c(\b)x_1^-}$ and consider the case $n=k+1$. 
Decomposing along the value of $\pi(1)$ one has
$$
\phi(k+1) = \sum_{j=1}^n\sum_{\pi:\pi(1)=j}e^{-c(\b) \ell(\pi)}.
$$
By the inductive assumption one has: if $j$ is such that $x_j>0$, then
$$
\sum_{\pi:\pi(1)=j}e^{-c(\b) \ell(\pi)}\leq e^{-c(\b)(x_j-x_1)}(1+\d)^k,
$$
while if $j$ is such that $x_j<0$ (and thus $x_1<0$), then
$$
\sum_{\pi:\pi(1)=j}e^{-c(\b) \ell(\pi)}\leq e^{-c(\b)(x_1^- + |x_j|)}(1+\d)^k.
$$
To conclude observe that if $x_1>0$ (and thus $x_j>0$ for all $j\geq 1$), then one has
$$
\phi(k+1)\leq \sum_{j=1}^ne^{-c(\b)(x_j-x_1)}(1+\d)^k\leq (1+\d)^{k+1}
$$
if $\d\geq \sum_{h=1}^\infty e^{-c(\b)h}$. If instead $x_1<0$ then 
$$
\phi(k+1)\leq (1+\d)^k\sum_{j=1}^n[e^{-c(\b)(x_j-x_1)}\IND(x_j>0) + e^{-c(\b)(x_1^- + |x_j|)}\IND(x_j<0)],
$$
which is bounded by $2\d(1+\d)^ke^{-c(\b)x_1^- }$. This implies the claim \eqref{asapa2}.
\end{proof}

We turn to the proof of Theorem \ref{th2p}. We need a version of Proposition \ref{propoz} for lattice paths. 
For integers $j\geq 0$,
define
\begin{align}\label{o+jp}
\O_{0,L}^{+,j}=\{\g\in\O_{0,L}:\; \g\geq -j
\}\,,\qquad \cZ^{+,j}_{0,L} = \sum_{\g\in\O_{0,L}} e^{-\b|\g|}\IND(\g\in\O^{+,j}_{0,L}),
\end{align}
where the condition $\g\geq - j $ means that all vertices of $\g$ have vertical coordinate at least $-j$. 
For the purpose of the recursive argument we have to consider also the number of external contacts of the path with the zero line. Namely, for any $\g\in \O_{0,L}$, define
\begin{align}\label{next}
N_{\rm ext}(\g)=\sum_{i\in\bbZ} \IND((i+\tfrac12)\in \g) \IND(i\notin [0,L-1]).
\end{align}
If $N(\g)$ denotes 
the number of internal contacts with level zero as in \eqref{clpin1p}, for $\e>0$ we write
\begin{align}\label{zo+jp}
\cZ^{\e,+,j}_{0,L} = \sum_{\g\in\O_{0,L}} e^{-\b|\g| + \e N(\g)+\e N_{\rm ext}(\g)}\IND(\g\in\O^{+,j}_{0,L}).
\end{align}
\begin{definition}\label{defhn2}
Fix the constant $C>0$ as in Proposition \ref{prop:3}, an integer $j\geq 0$ and $\b>0$. 
Let $\si^2=e^{-\b}$, and define $L_0 = \frac12C(j+1)^2/\si^2$, $L_n: = 2^n L_0$. For $n\in\bbN$, $\d>0$, and $b>0$, let $\cH_n(\d,b)$ denote the following statement: 
\[
\cZ^{\e,+,j}_{0,L}\leq (1+\d)  \cZ_{0,L} \quad \forall L\in [1,L_n],
\]
where $\e=b\si^2/(j+1)$.
\end{definition}

\begin{lemma}\label{lempropoz}
For any $\d>0$, there exist constants $b_0>0,\b_0>0$ such that $\cH_1(b,\d)$ holds for all $b \leq b_0$ and  $\b\geq \b_0$. 
\end{lemma}
\begin{proof}
Dropping the wall constraint, we have \begin{align*}
\cZ^{\e,+,j}_{0,L}& \leq \sum_{\g\in\O_{0,L}} e^{-\b|\g| + \e N(\g)+\e N_{\rm ext}(\g)} \\ &= \cZ_{0,L}\,
\bbE_{0,L}\big[e^{\e N(\g) + \e N_{\rm ext}(\g)}\big].
\end{align*}
From Corollary \ref{cor:1} we know that $\bbE_{0,L}\big[e^{2\e N_{\rm ext}(\g)}\big]\leq 1+u_\b$
for some constant $u_\b\to 0$ as $\b\to\infty$. Thus, using Schwarz' inequality we have 
$$
\cZ^{\e,+,j}_{0,L}\leq (1+u_\b)^{1/2}\cZ_{0,L}\,
\Big(\bbE_{0,L}\big[e^{2\e N(\g)}\big]\Big)^{1/2}.
$$
Now, if $L\leq L_1$ then $2\e   \leq 2 b \sqrt {C}\si /\sqrt L$.  Therefore, Lemma \ref{smallscalep} implies 
$$\bbE_{0,L}\big[e^{2\e N(\g)}\big]\leq 1+2 \sqrt {C}c b.$$ It remains to take  $b,\b$ such that $ (1+u_\b)^{1/2}(1+ 2c\sqrt {C} b)^{1/2}\leq 1+\d$.
\end{proof}
The next proposition establishes that $\cH_n(\d,b)$ holds in fact for all $n\in\bbN$ if we take e.g.\ $\d=1$, $b$ small enough, and $\b$ large enough. 
\begin{proposition}\label{lpropoz}
There exist constants $b>0,\b_0>0$ such that if $\e\leq b \,e^{-\b}/(j+1)$, then uniformly in $L,j$ and $\b\geq\b_0$ one has 
\begin{align}\label{lpropoz1}
\cZ^{\e,+,j}_{0,L} \leq 2 \cZ_{0,L}. 
\end{align}
\end{proposition}
\begin{proof}
We repeat the very same argument of Lemma \ref{indu}. The variable $\xi(\g)$ is now defined as the horizontal coordinate of the last contact of $\g$ with the zero line in the interval $[0,L/2]$, while $\eta(\g)$ denotes the horizontal coordinate of the first contact with the zero line in the interval $(L/2,L]$. 
Equation \eqref{indu2} now becomes: for any $L\in (L_n,L_{n+1}]$, 
\begin{align}\label{lindu2}
\cZ^{\e,+,j}_{0,L} \leq  (1+\d)^2 e^{2\e}\!\!\!\sum_{x\leq L/2}\sum_{y>L/2}
 \cZ_{0,x}\Xi^\e_{x,y}\cZ_{y,L},
\end{align}
where 
\begin{align}\label{lindu10}
\Xi^\e_{x,y}:=\sum_{\g\in\O((x,0),(y,0))}e^{-\b|\g| + \e N_{\rm ext}(\g)}\IND(N(\g)=0,\g\geq -j).
\end{align}
Note that the above sum includes the contribution of possible external zeros. This comes from the event that  the path $\g$ intersects the zero line outside of the interval $[x,y]$.  Moreover, in \eqref{lindu2} we have used the inductive assumption $\cH_n(\d,b)$. Notice that due to the possible presence of overhangs in the path $\g$ it is crucial to take into account the external contacts $N_{\rm ext}$ in order to exploit the induction. 

From Corollary \ref{cor:1} in the appendix we know that $\Xi^\e_{x,y}\leq (1+u_\b)\Xi^0_{x,y}$ for some constant $u_\b\to 0$ as $\b\to\infty$. 
Moreover, Proposition  \ref{prop:1} in the appendix shows that $\Xi^0_{x,y}\leq (1+u_\b)\wt\Xi^0_{x,y}$, where $\wt\Xi^0_{x,y}$ denotes the sum $\Xi^0_{x,y}$ in \eqref{lindu10} restricted to paths that are regular at the endpoints, see Definition \ref{regular}.  From Corollary \ref{cororeg} one has $\cZ_{0,x}\leq (1+u_\b) \widehat\cZ_{0,x}$ and $\cZ_{y,L}\leq (1+u_\b) \widehat\cZ_{y,L}$ where $\widehat \cZ_{u,v}$  denotes the partition function $\cZ_{u,v}$  restricted to paths that start and end with horizontal edges and that are regular at the points $x=u+1$ and $x=v-1$. 
Notice that if $\g_1$ is a path appearing in $ \widehat\cZ_{0,x}$, $\g_2$ is a path from $\wt\Xi^0_{x,y}$, and $\g_3$ is a path from $\widehat\cZ_{y,L}$, then the composition $\g_1\circ\g_2\circ\g_3$ defines a valid path from $0$ to $L$. That is, the above restrictions  allow us to avoid complications due to the self-avoiding constraint when we reconstruct the global partition function $\cZ_{0,L}$. 

From \eqref{lindu2} we obtain, for some constant $v_\b\to 0$ as $\b\to\infty$, 
\begin{align}\label{lindu200}
\cZ^{\e,+,j}_{0,L} \leq  (1+\d)^2 (1+v_\b)e^{2\e}\!\!\!\sum_{x\leq L/2}\sum_{y>L/2}
 \widehat\cZ_{0,x}\wt\Xi^0_{x,y}\widehat\cZ_{y,L}.
\end{align}
Define $\g_{\rm max}(\ell)$ as the maximal vertical height of $\g$ at points with horizontal coordinate $\lfloor\ell\rfloor$. Then 
it is not hard to see that \eqref{lindu200} implies 
\begin{align}\label{lindu3}
\frac{\cZ^{\e,+,j}_{0,L}}{\cZ_{0,L}} \leq  (1+\d)^2 (1+v_\b)e^{2\e}\, \bbP_{0,L}(\g_{\rm max}(L/2)\geq -j).
\end{align}
As in the proof of Proposition \ref{propoz}, it remains to show that $\bbP_{0,L}(\g_{\rm max}(L/2)\geq -j)\leq 1-\d_0$ for some $\d_0>0$,
uniformly in $L\geq L_1=C(j+1)^2e^\b$, $j\geq 0$ and $\b>\b_0$.  This follows from Proposition \ref{prop:3} (b) in the appendix.
\end{proof}

\begin{proof}[Proof of Theorem \ref{th2p}]
We proceed exactly as in \eqref{bot0}. Thus,
\begin{align}\label{bot0p}
\cZ^{\Phi,+,0}_{0,L}&
\leq
\sum_{j=0}^\infty\r_j
\sum_{\g\in\O^{+,0}_{0,L}}
e^{-\b|\g|}e^{\k_j N_j(\g)} ,
\end{align}
where $\k_j:=\e_j/\r_j = b\si^2/(j+1)$, and $\sum_{j=0}^\infty \r_j=1$.
The estimate \eqref{bot1} now takes the form
\begin{align}\label{tkf}
\sum_{\g\in\O^{+,0}_{0,L}}
e^{-\b|\g|}e^{\k_j N_j(\g)} \leq 4 \cZ_{0,L}.
\end{align}
To prove \eqref{tkf}, thanks to Proposition \ref{lpropoz} the same reasoning as in  \eqref{ubot} can be applied with minor modifications. The only difference is that to conclude one needs to handle the partition functions in  \eqref{ubot} with some care in order to restore the final partition function $\cZ_{0,L}$. More precisely, fix $u,v\in\bbZ$, with $0\leq u\leq v\leq L$; let $\hat \cZ^{+,0}_{0,u}$ denote the partition function corresponding to paths $j\geq \g\geq 0$ from $(\tfrac12,0)$ to $(u+\tfrac12,j)$ that never touch level $j$ before the endpoint; similarly, let  $\hat \cZ^{+,0}_{v,L}$ denote the partition function corresponding to paths $j\geq \g\geq 0$ from $(v-\tfrac12,j)$ to $(L-\tfrac12,0)$ that never touch level $j$ after the staring point.
Then, reasoning as in \eqref{bot012} and \eqref{ubot}, the sum in \eqref{tkf} is bounded above by
\begin{align}\label{tkf2}
\cZ_{0,L} + e^{2\k_j}
\sum_{u\leq v} \hat \cZ^{+,0}_{0,u} \cZ^{\e,+,j}_{u,v}\hat \cZ^{+,0}_{v,L},
\end{align}
where $\cZ^{\e,+,j}_{u,v}=\cZ^{\e,+,j}_{0,v-u} $ is defined as in \eqref{zo+jp} above. Note that for this estimate to hold it is crucial that $\cZ^{\e,+,j}_{u,v}$ takes into account the interaction with external zeros. By Proposition \ref{lpropoz} one has $\cZ^{\e,+,j}_{u,v}\leq 2\cZ_{u,v}$. Moreover, 
let $\wt  \cZ^{+,0}_{0,u},\wt  \cZ^{+,0}_{v,L}$ denote 
the partition functions $\hat 
\cZ^{+,0}_{0,u}$ and $\hat 
\cZ^{+,0}_{v,L}$ respectively defined with the restriction that the path is regular at the endpoints, see Definition \ref{regular}. Then, from Proposition \ref{prop:1} in the appendix one has that 
$\hat \cZ^{+,0}_{0,u}\,\hat \cZ^{+,0}_{v,L}
\leq (1+u_\b)\wt \cZ^{+,0}_{0,u}\, \wt \cZ^{+,0}_{v,L}$,
for some $u_\b\to 0$ as $\b\to\infty$. 
Furthermore, from Corollary \ref{cororeg} one has $\cZ_{u,v}\leq (1+u_\b) \widehat\cZ_{u,v}$,  
where $\widehat \cZ_{u,v}$  denotes the partition function $\cZ_{u,v}$  restricted to paths that start and end with horizontal edges and that are regular at the points $x=u+1$ and $x=v-1$. 
It follows that 
\begin{align}\label{tkf1}
\hat \cZ^{+,0}_{0,u}\,\cZ_{u,v}\,\hat \cZ^{+,0}_{v,L}
\leq (1+v_\b)\wt \cZ^{+,0}_{0,u}\,\widehat \cZ_{u,v}\, \wt \cZ^{+,0}_{v,L},
\end{align}
for some $v_\b\to 0$ as $\b\to\infty$
As discussed before Eq.\ \eqref{lindu200}, the regularity constraint at the endpoints together with the restriction to horizontal edges allows us to reconstruct the partition function $\cZ_{0,L}$ as an upper bound  
so that, taking $\b$ large enough, one can conclude in the same way as in the proof of  \eqref{bot1}. 
\end{proof}

\appendix
\section{}
\label{A}
In this first appendix we prove few technical results which, roughly speaking,
say that, even in the presence of a wall, the lattice path ensemble for $\b$ large is likely to intersect
a given vertical line only once, not to make  excursions to the
left of its starting point (or to the right of the final point) and to start and to end with a horizontal bond. These very
intuitive results are  useful when trying to concatenate together
different pieces of the path in our recursive method.   

Consider the lattice path model defined in Section \ref{latticepaths}.
We introduce a bit more notation. 
Given
$H\in \bbZ$ we will denote by $\O_{0,L}^H$ the space
of self-avoiding lattice paths $\g$ connecting $(1/2,0)$ to $(L-1/2,H)$. The
corresponding ensemble will be denoted by $\bbP_{0,L}^H$. If $H=0$ we
will simply write $(\O_{0,L},\bbP_{0,L})$ as usual. 
The vertical line
through the point $(u,0)$, $u\in \bbZ$, will be denoted by 
$\cL_u$ and the cardinality of a finite set $S$ will be denoted
by $|S|$. We write
$\g\geq - j$ if the height of $\g\in \O_{0,L}^H$ is everywhere at least
$-j$. Below, $j$ is always a non-negative integer. 
\begin{definition}[Local regularity]
\label{regular}
Given $u\in [1,L-1]\cap \bbZ$ we say that $\g\in\O_{0,L}^H$ is {\em regular} at $u$ if $\g$ intersects \emph{only once} the line $\cL_u$, i.e.\ if $|\g\cap \cL_u|= 1$. For $u\in\{0,L\}$, we say that $\g\in\O_{0,L}^H$ is regular at $u$ if $\g$ \emph{does not} intersect $\cL_u$, i.e.\ if $|\g\cap \cL_u|= 0$. If $\g\in\O_{0,L}^H$ is regular at both $0$ and $L$ we say that $\g$ is regular at the endpoints.
\end{definition}
As an example, both paths in Figure \ref{fig:paths} are regular at the endpoints, the left path is regular at $u=2$, while the right path is not. 
\begin{proposition}
\label{prop:1} 
We have \begin{equation}
  \label{eq:1}
\lim_{\b\to +\infty}\sup_{j,L}\,\maxtwo{H\geq -j}{ |H|\le L}\,\max_{u\in[0,L]}
\bbP_{0,L}^H\left(\g \text{ is not regular at $u$}\mid \g \geq -j\right)=0.   
\end{equation}
\end{proposition}
  For simplicity we only treat the case $u=0$ but the same strategy
  with minor modifications works for other values of $u$. Let $
\hat \bbP_{0,L}^H (\cdot)=\bbP_{0,L}^H(\cdot\mid \g\cap
\cL_L=\emptyset)$ and let $\cA$ denote the event that $\g\in \O_{0,L}^H$ is not regular at $0$. 
We will first show that
\begin{equation}
  \label{eq:2}
  \lim_{\b\to +\infty}\sup_{j,L}\,\maxtwo{H\geq -j}{ |H|\le L}
\hat \bbP_{0,L}^H \left(\cA\mid \g \geq -j\right)=0.   
\end{equation}
Later on we will show that \eqref{eq:2} implies the same bound for $ \bbP_{0,L}^H $. 

To prove \eqref{eq:2}
fix $\d\le 1/10$ (independent of $\b$) and  let $\cG_L$ be the event that there exists 
$y\in [L/2- \d^2 L^{1-\d},L/2 +\d^2 L^{1-\d}]\cap \bbZ$ such that $\g\cap
\cL_y=(y,H_y)$  and $H_y$
is within  
  $\d^2L^{1-\d}$ from the average height $\bar H_y:=Hy/L$.
\begin{lemma}
\label{A.2}
For all $\b$ large enough and uniformly in $j\geq 0$, $|H|\le 3L/2$ and $L$,  
\[
\hat \bbP_{0,L}^H \left(\cG_L\mid \g \geq -j\right)\geq 1-\exp(-c L^{1-2\d}),
\]  
for some constant $c=c(\d,\b)$ with $\lim_{\b\to \infty}c(\d,\b)=+\infty$. The same bound applies to $\bbP_{0,L}^H \left(\cG_L\mid \g \geq -j\right)$. 
\end{lemma}
\begin{proof}[Proof of the Lemma]
We prove the lemma for $\hat \bbP_{0,L}^H$ but the same arguments apply to $\bbP_{0,L}^H$. Clearly,
\[
\hat \bbP_{0,L}^H\left(\cG^c_L\mid
  \g\geq -j\right)\le \hat \bbP_{0,L}^H
\left(\cG^c_L\right)/
\hat \bbP_{0,L}^H\left(\g\geq -j\right).
\]
Using \cite{DKS}*{Section 4.14} the probability that there exists
$y\in [L/2- \d^2 L^{1-\d},L/2 +\d^2 L^{1-\d}]\cap \bbZ$ such that
$\g\cap
\cL_y$ contains a point whose height differs from $\bar H_y$ by more than
$\d^2L^{1-\d}$ is smaller than $e^{-c 
  L^{1-2\d}}$ for some constant $c=c(\d)$.
If for all $y\in [L/2- \d^2 L^{1-\d},L/2 +\d^2 L^{1-\d}]\cap \bbZ$ the set $\g\cap \cL_y$
is \emph{not} a singleton and it
is contained in the interval $[\bar H_y\pm \d^2L^{1-\d}]$, then  in the
  interval $[L/2- \d^2 L^{1-\d},L/2 +\d^2 L^{1-\d}]$ the path $\g$ has
  length at least $10\d^2L^{1-\d}$ {\it i.e.}  an excess
 length (w.r.t. to its minimal length) of at least 
$(6-2H/L)\d^2L^{1-\d}\geq 3\d^2 L^{1-\d}$. Therefore a 
Peierls argument shows that the above event has probability not larger
than $e^{-c'\b
  L^{1-\d}}$ for some constant $c'=c'(\d)$. In conclusion, by
renaming the constants if necessary,  $\hat \bbP_{0,L}^H\left( \cG^c_L\right)\le e^{-c 
  L^{1-2\d}}$
for some constant $c=c(\d,\b)$ diverging as $\b\to +\infty$. 

We conclude with a rough lower bound on $\hat
\bbP_{0,L}^H\left(\g\geq -j\right) $ of the form 
\begin{equation}
  \label{eq:13}
 \hat \bbP_{0,L}^H\left(\g\geq -j\right)=\sum_{\g\geq -j\atop \g\cap
   \cL_L=\emptyset}e^{-\b|\g|}/\sum_{\g\cap \cL_L=\emptyset}e^{-\b|\g|}\ge
 e^{-c''\b L^{1-3\d}}
\end{equation}
for some constant $c''=c''(\d)$. To prove this, let $\ell =\lfloor L^{1-3\d}\rfloor$
with $\d<1/6$. We can restrict the sum in the
numerator above to paths which, while staying above $-j$ and never
intersecting $\cL_L$, first go straight to the
point $(0,\ell)$, then reach the point $(L,H+\ell)$ and finally go
straight to the point
$(L,H)$. Since $\ell\gg
L^{1/2}$, \cite{DKS}*{?} implies that the extra constraint of staying
above level $-j$ is irrelevant for this restricted sum which is
therefore greater
than e.g. $e^{-3\b \ell}\sum_{\g\cap \cL_L=\emptyset}e^{-\b|\g|}$. Hence the claimed bound.
\end{proof}
We return  to the proof of \eqref{eq:2} let 
\[
p_L:= \suptwo{L'}{|L-L'|\le L^{1-\d}}\sup_{j}\,\maxtwo{H\geq -j}{
  |H|\le L+L^{1-\d}}\hat \bbP_{0,L'}^H\left(\cA\mid \g \geq -j\right).
\]
We claim that
\begin{equation}
  \label{eq:3}
p_L\le p_{L/2} + \exp(-c L^{1-2\d})
\end{equation}
with $c=c(\d)$.
Fix $L'\in [L-L^{1-\d},L+L^{1-\d}]$. On the event $\cG_{L'}$ we can condition
on the rightmost point $\xi\in [L'/2-\d^2
(L')^{1-\d},L'/2+\d^2 (L')^{1-\d}]\cap \bbZ$ satisfying the
requirements of $\cG_{L'}$ and on the height $H_\xi$ of the path there. By construction, for $\d$ small enough and
$L$ large enough independent of $\b$,
\begin{equation}
  \label{eq:9}
\xi\in [L/2-(L/2)^{1-\d},L/2+(L/2)^{1-\d}]\quad \text{and}\quad 
|H_\xi|\le (1+\left(L/2\right)^{-\d})L/2.   
\end{equation}
Notice that the law of the part of the path $\g$ joining the origin to
$(\xi-1/2,H_\xi)$ 
is exactly $\hat\bbP_{0,\xi}^H$. Hence, using Lemma \ref{A.2} and
\eqref{eq:9}, we get 
\begin{align*}
\hat \bbP_{0,L'}^H\left(\cA\mid \g \geq -j\right)&\le
\max_{\xi,H_\xi\sim \cG_{L'}}\hat \bbP_{0,L'}^H \left(\cA\mid \xi,H_\xi,\g \geq -j\right)+\hat
\bbP_{0,L'}^H\left(\cG_{L'}^c\mid \g \geq -j\right)\\
&\le p_{L/2}+e^{-c L^{1-2\d}},
\end{align*}
i.e. \eqref{eq:3}. To finish the proof of \eqref{eq:2} it is sufficient to observe that,
for any fixed  $L_0$, 
\begin{equation*}
  \lim_{\b\to +\infty}\sup_{j}\,\maxtwo{H\geq -j}{ |H|\le L_0}\hat \bbP_{0,L_0}^H \left(\cA\mid \g \geq -j\right)=0,   
\end{equation*}
because the event $\cA$ forces the path to have an excess length
(w.r.t. to the minimal lenght) of
at least $2$. Hence \eqref{eq:3} implies that, for any $L_0$, 
\[
  \limsup_{\b\to +\infty}\ \sup_{j,L}\,\maxtwo{H\geq -j}{ |H|\le 2 L}
\hat \bbP_{0,L}^H\left(\cA\mid \g \geq -j\right)\le
\sum_{n\geq 0}e^{-c(L_0\,2^n)^{1-2\d}}   
\]
and \eqref{eq:2} follows.

We finally observe that \eqref{eq:1} follows at once from
\eqref{eq:2}. Fix some large $L_0$ independent of $\b$. For $L\le L_0$
a simple Peierls argument shows that
\[
  \lim_{\b\to +\infty}\sup_{j}\,\maxtwo{H\geq -j}{ |H|\le 2 L}\bbP_{0,L}^H \left(\cA\mid \g \geq -j\right)=0.
\] 
For $L>L_0$ write as before
\[
\bbP_{0,L}^H\left(\cA\mid \g \geq -j\right)\le
\max_{\xi,H_\xi\sim \cG_{L}}\bbP_{0,L}^H\left(\cA\mid \xi,H_\xi,\g \geq -j\right)+
\bbP_{0,L}^H \left(\cG_{L}^c\mid \g \geq -j\right)
\]
In the first term in the r.h.s. the path to the left of $\cL_\xi$ has
exactly the distribution $\hat \bbP_{0,\xi}^H$. Hence its $\b\to \infty$ limit is zero by
\eqref{eq:2}. The second term in the r.h.s. is smaller than $e^{-c
  L^{1-2\d}}\le e^{-c
  L_0^{1-2\d}}$ by the previous lemma. Since $L_0$ was arbitrary the result follows.
This ends the proof of Proposition \ref{prop:1}.

\begin{proposition}
\label{prop:4}
Given $\g\in \O_{0,L}^H$ let $X(\g)$ be the position of the leftmost
zero of $\g$. There exist $C$  and
$\b_0$ such that, for all $\b\geq \b_0$ and all $\ell>0$,
\begin{equation}
  \label{eq:11}
 \sup_{j,L}\,\maxtwo{H\geq -j}{ |H|\le L}\,\bbP_{0,L}^H\left(X(\g)= -\ell\mid
   \g \geq -j\right)\le Ce^{-m_\b\, \ell},
\end{equation}
with $\lim_{\b\to \infty}m_\b=+\infty$.
\end{proposition}
\begin{corollary}
  \label{cor:1}
Let $N_{\rm ext}(\g)$ denote the external zeros as defined in \eqref{next}.
Then, for any fixed constant $a>0$:
\begin{equation}
  \label{eq:next1}
\lim_{\b\to +\infty}\ \sup_{j,L}\ \maxtwo{H\geq -j}{ |H|\le
  L}\,\bbE_{0,L}^H\left(e^{\,a N_{\rm ext}(\g)}\mid \g\geq -j\right)=1
\end{equation}
\end{corollary}
\begin{proof}[Proof of the Corollary]
We may write $N_{\rm ext}=N_-(\g)+N_+(\g)$, where $N_+(\g)$ (resp.\ $N_-(\g)$) denotes the number of zeros to the right (resp.\ left) of the interval $[0,L]$. By Schwarz' inequality and using symmetry it suffices to prove the statement \eqref{eq:next1} with $N_{\rm ext}$ replaced by $N_-(\g)$.  Clearly $N_-(\g)\le |X(\g)|$. The statement then follows at once from
Proposition \ref{prop:4}.  
\end{proof}
\begin{proof}[Proof of Proposition \ref{prop:4}]
Fix $\d\le 1/10$ and assume first that $\ell\geq L^{1-\d}$. In this
case, using \eqref{eq:13},  it suffices to prove the required estimate for the
unconditioned probability. If by
traveling along the path the last intersection with the column $\cL_1$ is at
height $Y$ with $|Y|\le \ell$ then the part of the path between the
origin and the vertex $(1,Y)$ has an excess length of at least $\ell$.
A standard Peierls argument shows that the unconditioned probability of such an
event is bounded from above by $e^{-m_\b\ell}$ with $\lim_{\b\to
  \infty}m_\b=+\infty$. If instead $|Y|\geq \ell$ we can appeal to the
large deviation bounds in \cite{DKS}*{Section 4} to get the same result.

Suppose now that
$\ell \le L^{1-\d}$. The argument leading to \eqref{eq:3}, 
with $L_0=\ell^{1/(1-2\d)}$,  implies the estimate 
\begin{gather*}
 \sup_{j,L}\,\maxtwo{H\geq -j}{ |H|\le L}\,\bbP_{0,\ell}^H\left(X(\g)= -\ell\mid
   \g \geq -j\right)\\\le 
\sup_{j}\,\maxtwo{H\geq -j}{|H|\le 2L_0}\,
\bbP_{0,L_0}^H\left(X(\g)= -\ell\mid \g \geq -j\right)+ \sum_{n\geq 0}e^{-c_\b(L_0\,2^n)^{1-2\d}}. 
\end{gather*}
Using the first part of the proof we get that both terms in the
r.h.s.\ above can be bounded from above by $e^{-m_\b\ell}$ with $\lim_{\b\to \infty}m_\b=+\infty$.
\end{proof}
The last result says that, under $\bbP_{0,L}$, the path $\g$ is likely
to start with a horizontal bond. 
\begin{lemma}
\label{prop:oriz} Let $f_0$ be the first edge of the path $\g\in
\O_{0,L}$. Then
\[
\lim_{\b\to \infty}\sup_L \bbP_{0,L}\left(\text{$f_0$ is vertical}\right)=0
\]  
\end{lemma}
\begin{proof}
We will first prove the result in the ``grand canonical'' ensemble  
$\{\cO_L,\cP_L\}$, where $\cO_L$ is the set of self-avoiding paths
starting at $(1/2,0)$  and ending at $(L-1/2,H)$ for \emph{some} $H\in \bbZ$ and
$\cP_L(\g)=e^{-\b|\g|}/\Xi_L$ with 
$\Xi_L=\sum_{\g\in \cO_L}e^{-\b|\g|}$. Indeed, if we decompose over the number of the first consecutive
vertical edges, we
immediately get that
$\cP_L\left(\text{$f_0$ is
    vertical}\right)\le 2\sum_{n\geq 1} e^{-\b n}$. 
    The result for the canonical ensemble  
$\bbP_{0,L}(\cdot)=\cP_L(\cdot\mid \g \text{ ends at zero height}) $ follows from Lemma
\ref{lem:equiv} in Appendix \ref{B}.
\end{proof}

From Lemma \ref{prop:oriz} and Proposition \ref{prop:1} one has the following
\begin{corollary}\label{cororeg}
Let $\widehat\cZ_{0,L}$ denote the partition function obtained by restricting $\cZ_{0,L}$ to paths $\g\in\O_{0,L}$ that are regular at both $x=1$ and $x=L-1$ and such that the first and last edge of $\g$ is horizontal. Then there exists $u_\b>0$ with $u_\b\to 0$ as $\b\to\infty$ such that $\cZ_{0,L}\leq (1+u_\b)\widehat  \cZ_{0,L}$ for any $L\geq 1$. 
\end{corollary}

\section{}
\label{B}
Here, we prove two estimates on moderate deviations for random walks and self-avoiding paths. In the
random walk setting 
the arguments that we use  are rather standard but we decided to detail
them, on one hand in order to pave the way for the self-avoiding paths setting, and
on the other hand to get estimates that hold uniformly on all scales
and for all random walks in our class. 
\begin{proposition} 
\label{prop:2}
There exists $C\geq 1$ such that, for any $\s^2\in (0,1/2]$ and for any
random walk kernel $p\in \cP(\sigma^2)$, the
following holds.\\
(a) For any $L$,
\begin{equation}
  \label{eq:4}
\frac{C^{-1}}{\max\left(1,\sqrt{\sigma^2L}\right)}\leq  Z_{0,L}\le \frac{C}{\sqrt{\sigma^2L}}.
\end{equation}
(b) For all $j\geq 0$ and for all $L\geq C
  (j+1)^2 /\sigma^2$, 
  \begin{equation}
    \label{eq:5}
\bbP_{0,L}(\g(L/2)\geq -j)\le 3/4,
  \end{equation}
  where $\g(L/2)=\g_{\lfloor L/2\rfloor}$.
\end{proposition}
\begin{proof}\ \\
(a) Let $S_n=X_1+\dots +X_n$, $S_0=0$, where $\{X_i\}_{i=1}^\infty$ are
i.i.d. random variables with law $p\in \cP(\sigma^2)$. In the sequel
we will write $f_n(t)$ for the characteristic function of
$S_n/\sigma_n$ where $\sigma_n=\sqrt{n} \sigma$. Clearly
$f_n(t)=\varphi(t/\sigma_n)^n$ where $\varphi(\cdot)$ is the
characteristic function of the variable $X_1$. Let $(S_n,S_n')$ be two independent copies of the same
random variable.
Using the identity
\[
\mathds{1}_{\{S_n=S_n'\}}= \frac{1}{2\pi \sigma_n}\int_{-\pi \sigma_n}^{\pi\sigma_n}e^{i\frac
  {t}{\sigma_n} (S_n-S_n')} dt 
\]
we can write
\begin{gather}
\label{eq:20}
\bbP(S_n=S_n') = 
\frac{1}{2\pi \sigma_n}\int_{-\pi \sigma_n}^{\pi\sigma_n} dt\
  |f_n(t)|^2
\end{gather}
Notice that $|\varphi(t/\sigma_n)|\le e^{\frac 12(|\varphi(t/\sigma_n)|^2-1)}$
and that, using the assumption $p\in\cP(\si^2)$, 
\begin{align*}
|\varphi(t/\sigma_n)|^2-1&= -\var\left(\cos(t
    X_1/\sigma_n)\right)-\var\left(\sin(t X_1/\sigma_n)\right)\\
&\le -2p(0)p(1)(1-\cos(t/\sigma_n))\\
&\le -c_0(1-\sigma^2)\frac{t^2}{2 n}\le -\frac{c_0}{4n}t^2
\end{align*}
for all $|t|\le \pi \sigma_n$. Above we used the formula 
$\var(Z)=\frac 12 \bbE((Z-Z')^2)$ where $Z,Z'$ are two independent
copies of the same random variable to get 
\[
\var\left(\cos(t X_1/\sigma_n)\right)+\var\left(\sin(t
  X_1/\sigma_n\right)
\geq p(0)p(1)\left[(1-\cos(t/\sigma_n))^2+\sin(t/\sigma_n)^2\right].
\]
In conclusion $|f_n(t)|\le e^{-c_0
  t^2/8}$ and
\[
\bbP(S_n=S_n')\le \frac{C}{\s_n}
\]
for some constant $C$ depending on $c_0$. 
To prove a lower bound we use 
\[|e^{it}-1-it + \frac{t^2}{2}|\le
\min\left(|t|^2, |t|^3/6\right)],
\] 
to
write
\[
|\varphi(t/\sigma_n)-1+ \frac{t^2}{2n}| \le
\frac{t^2}{\sigma_n^2}\bbE\left(X_1^2\wedge \frac{|t|}{6
    \sigma_n}|X_1|^3\right)\le c \frac{|t|^3}{6 n \sigma_n},
\] 
where we used the assumptions $\bbE(X_1)=0$ and $\bbE(|X_1|^3)\le c \sigma^2$. Thus, if
$\sigma_n\geq 1$ and by choosing $\d$ small enough independent of $\sigma^2$, we
obtain
\begin{gather*}
\bbP(S_n=S_n') = 
\frac{1}{2\pi \sigma_n}\int_{-\pi \sigma_n}^{\pi\sigma_n} dt\
  |f_n(t)|^2\\
\geq \frac{1}{2\pi \sigma_n}\int_{-\d}^{\d} dt\
  |f_n(t)|^2\geq \frac{1}{2\pi
    \sigma_n}\int_{-\d}^{\d} dt\ e^{-t^2/4}\geq \frac{C_\d}{\sigma_n}.
\end{gather*}
If instead $\sigma_n\le 1$ we simply write
\[
\bbP(S_n=S_n')\geq \bbP(S_n=0)^2\geq p(0)^{2 n}\geq (1-\sigma^2)^{2n}\ge
e^{-4}.
\]
Equation \eqref{eq:4} follows if we observe that (without loss of
generality we assume $L$ even)
\[
Z_{0,L}= \bbP(S_{L/2}=S'_{L/2}).
\] 
We turn to  the proof of part (b). With the previous notation we can write ($n=L/2$)
\begin{gather}
\bbP(\g(n)\geq -j\tc \g(2n)=0) =\bbP(S'_n\geq -j\mid S_n=S_n')\nonumber\\ 
=\frac{\int_{-\pi \sigma_n}^{\pi\sigma_n} dt\
  f_n(t)\bbE\left(\mathds{1}_{S'_n\geq -j}e^{-i\frac{t}{\sigma_n}S'_n}\right)}{\int_{-\pi \sigma_n}^{\pi\sigma_n}dt\ |f_n(t)|^2}\nonumber\\
\label{eq:10}
\le \frac{\int_{-\pi \sigma_n}^{\pi\sigma_n} dt\
  |f_n(t)|}{\int_{-\pi \sigma_n}^{\pi\sigma_n}dt\ |f_n(t)|^2}\  \bbP(S_n'\geq -j).
\end{gather}
We now claim that, given $0<\d\le 0.02$, we can choose $C=C(\d)$ so large
that, for $n\geq C/\sigma^2 $,
\begin{equation}
  \label{eq:6}
\frac{\int_{-\pi \sigma_n}^{\pi\sigma_n} dt\
  |f_n(t)|}{\int_{-\pi \sigma_n}^{\pi\sigma_n}dt\ |f_n(t)|^2}\le
\sqrt{2}+\d, 
\end{equation}
and for $n\geq C(j+1)^2/\sigma^2 $
\begin{equation}
  \label{eq:7}
 \bbP(S_n'\geq -j)\le 
\frac 12 + \d.
\end{equation}
These bounds imply \eqref{eq:5}.

We begin by discussing \eqref{eq:6}. 
Consider first $\int_{-\pi \sigma_n}^{\pi\sigma_n} dt\
|f_n(t)|$. For any $A>0$ we write
\begin{gather*}
 \int_{-\pi \sigma_n}^{\pi\sigma_n} dt\ |f_n(t)| \le \int_{-\infty}^\infty dt\
 e^{-t^2/2}+ J_1+J_2+J_3= \sqrt{2\pi}+ J_1+J_2+J_3,
\end{gather*}
where 
\begin{align*}
  J_1= \int_{-A}^{A} dt\ |f_n(t)-e^{-t^2/2}|,\quad J_2= \int_{|t|\geq A} dt\ e^{-t^2/2},\quad
J_3= \int_{A\le |t| \le \pi \sigma_n}dt\ |f_n(t)|.
\end{align*}
Using the bound $|f_n(t)|\le e^{-c_0t^2/4}$ we can always choose $A$
in such a way that $J_2,J_3\le \d/6$. Given $A$, we can use as
before the second order Taylor expansion and choose $C$ so large
(independent of $\sigma^2$) that
$J_1\le \d/6$. In conclusion
\[
 \int_{-\pi \sigma_n}^{\pi\sigma_n} dt\ |f_n(t)| \le \sqrt{2\pi}+
 \d/2.
\]
To lower bound $\int_{-\pi \sigma_n}^{\pi\sigma_n}dt \,|f_n(t)|^2$ we may simply restrict the integral to $|t|\le A$ and get
\begin{align*}
 \int_{-\pi \sigma}^{\pi\sigma}dt\ |f(t)|^2 &\geq \int_{-A}^{A}dt\  e^{-t^2} -
 \int_{-A}^{A} dt\ |\ |f(t)|^2-e^{-t^2}|\\
&= \sqrt{\pi} -\int_{|t|\geq A} dt\  e^{-t^2}-
 \int_{-A}^{A} dt\ |\ |f(t)|^2-e^{-t^2}|
\end{align*}
and choose again $A,C$ large enough to make the two error terms smaller
than $\d/2$. Thus the l.h.s.\ of \eqref{eq:6} is smaller than
$\frac{\sqrt{2\pi}+\d/2}{\sqrt{\pi}-\d/2}\le \sqrt{2}+\d$ for $\d\le 1$. 

Next we prove \eqref{eq:7}. The Berry-Essen theorem (see e.g. \cites{Feller})
gives
\[
\sup_x |\bbP(S_n/\sigma_n\le x)-\Phi(x)| \le 3\frac{\bbE(|X_1|^3)}{\sigma^3\sqrt{n}},\]
where $\Phi(\cdot)$ is the distribution function of
  the standard normal.
Thus, using $\bbE(|X_1|^3)\le c\sigma^2$, 
\[
\bbP(S_n\geq -j)\le \Phi(j/\sigma_n) + \frac{c}{2\sigma_n}\le 1/2+\d
\]
for all $n\geq C (j+1)^2/\sigma^2$ with $C$ large enough depending on $\d$.
\end{proof}

Proposition \ref{prop:2} holds also for the ensemble of  self-avoiding
paths introduced in Section \ref{latticepaths}. We need few additional notation. We will denote by $\g_{\rm max}(L/2)$ the highest intersection of 
$\g\in \O_{0,L}$
with the vertical line through the point $(\lfloor L/2\rfloor,0)$. Recall that
$\cZ_{0,L}=\sum_{\g\in \O_{0,L}}e^{-\b|\g|}$ and let $\Xi_L$ be the partition
function of the
grand canonical ensemble $\{\cO_L,\cP_L\}$ defined in the proof of
Lemma \ref{prop:oriz}. Finally, denote by
$\{\hat \cO_L,\hat \cP_L\}$ the restricted grand canonical ensemble in which the
intersection of the path $\g\in \cO_L$ with the vertical line through
the point $(L,0)$ is empty. 
\begin{proposition}
\label{prop:3}
There exists $C$  and $\b_0$ such that, for all $\b\geq \b_0$, the
following holds. \\
(a) There exists $\hat \s_\b$ with $|\hat\s^2_\b- \frac{1}{\cosh(\b)-1}|\le e^{-2(\b-\b_0)}$ such that, for all $L\geq 1$, 
\begin{equation}
  \label{eq:12}
\frac{C^{-1}}{\max\left(1,\sqrt{\hat \sigma_\b^2L}\right)}\leq  \frac{\cZ_{0,L}}{\Xi_L}\le \frac{C}{\sqrt{\hat\sigma_\b^2L}}.
\end{equation}
(b) For all $j\geq 0$ and all $L\geq C
  (j+1)^2 e^\b$, 
  \begin{equation}
    \label{eq:8}
\bbP_{0,L}(\g_{\rm max}(L/2)\geq -j)\le 3/4.
  \end{equation}
\end{proposition}
Before proving the lemma we recall some key results from
\cite{DKS}. For any $\g\in \cO_L$ let $h_L(\g)$ be the height of its final point. Let also $\s^2_L,f_L(t)$ be the
  variance of $h_L(\g)$ and the characteristic function of $h_L(\g)/\s_L$
  respectively in the
  ensemble $\cP_L$. 
\begin{lemma}[\cite{DKS}*{Sections 4.9,4.10,4.10.20,4.10.29}]
\label{lemma:DKS}
There exists $\b_0>0$ such that, for all $\b\geq \b_0$, the following
holds.
\begin{enumerate}[(a)]
\item There exists
$\hat\s(\b)$ with $|\hat \s^2(\b)-\frac{1}{\cosh(\b)-1}|\le
e^{-2(\b-\b_0)}$ and such that 
\[|\s_L^2 -\hat \s(\b)^2 L|\le
e^{-2(\b-\b_0)}\quad \forall L.
\]
\vskip 0.3cm
\item There exist two constants $\l$ and $\a\le \pi$ independent of $\b,L$ such that
\begin{enumerate}[(i)]
\item $\log(f_L(t))= -\frac{t^2}{2} + \frac{t^3}{6}\frac{R_L(t)}{\s_L^3}$
  for all $t\in \bbR$ 
  with $\sup_t |R_L(t)|\le \l \s_L^2\,$.
\item $|f_L(t)|\le e^{-t^2/4}$ for all $|t|\le \a\, \s_L$.
\item $|f_L(t)|\le (1-\frac{\a^2}{4e^\b})^{L}\le e^{-c(\a)\s_L^2}$ for
  all $|t|\in (\a\s_L,\pi\s_L]$.
\end{enumerate}
\vskip 0.3cm
\item $1/2\le\Xi_L/(e^{-\b}\Xi_{L/2}^2)\le 1$ for $L$ even and
  $1/2\le \Xi_L/\Xi_{\lceil L/2\rceil}^2\le 1$ for $L$ odd.
\end{enumerate}
Similar bounds hold for the quantities $\hat \s^2_L,\hat f_L(t)$ computed in the restricted ensemble
$\hat\cP_L$.\end{lemma}

\begin{proof}[Proof of Proposition \ref{prop:3}] Without loss of generality
  we assume that $L$ is odd and we let $n=\lceil L/2\rceil$.
 
Using (c) of  Lemma \ref{lemma:DKS} it is enough to prove
\eqref{eq:12} for the ratio $\cZ_{0,L}/\Xi_{n}^2$. Using
Proposition \ref{prop:1}, for any $\b$ large enough uniformly in $L$, $\cZ_{0,L}\le 2\cZ^{{\rm reg}}_{0,L}$ where 
$\cZ^{{\rm reg}}_{0,L}:=\sum_{\g\in \O^{{\rm reg}}_{0,L}}e^{-\b|\g|}$ and 
$\O^{{\rm reg}}_{0,L}\subset \O_{0,L}$ consists of those paths that are regular at $n$.
In turn the ratio $\cZ^{{\rm reg}}_{0,L}/\Xi_{n}^2$ is at most
the probability (in the $\hat\cP_{n }$ ensemble) that
two independent copies of $\g\in \hat\cO_{n}$ have the same final
height, a quantity which can be written (cf. \eqref{eq:20}) 
\begin{equation}\label{eqxx}
\hat\cP_{n}\otimes
\hat\cP_{n} (h_n(\g_1)=h_n(\g_2))= \frac{1}{2\pi\s_{n}}\int_{-\pi \s_{n}}^{\pi \s_{n}}dt\, |\hat f_{n}(t)|^2.
\end{equation}
The desired upper bound now follows at once from (a) and (b) of Lemma \ref{lemma:DKS}. 

Next we lower bound
$\cZ_{0,L}$ by $\cZ^{{\rm reg}}_{0,L}$. Any $\g\in \O^{{\rm reg}}_{0,L}$ can be seen as formed by 
two paths $\g_i$, $i=1,2$, the first one starting from the origin and running
forward and the second one starting from $(L,0)$ and
running backward, until they hit $\cL_n$ at the same
height. These two paths are i.i.d with law $\hat \cP_{n}(\g)$. Thus 
\[
\frac{ \cZ^{\rm reg}_{0,L}}{\Xi_{n}^2}= \hat \cP_{n}\otimes
\hat\cP_{n} (h_n(\g_1)=h_n(\g_2))\left(\frac{\hat\Xi_{n}}{\Xi_{n}}\right)^2.
\]
Using (b) of Lemma \ref{lemma:DKS}, \eqref{eqxx} is bounded from below by
$1/C_1\max\left(1,\sqrt{\sigma_{n}^2}\right)$ for some large
$C_1$ independent of $\b$ as in Proposition \ref{prop:2}. The ratio $\hat\Xi_{n}/\Xi_{n}$ is
greater than e.g. $1/2$ using \cite{DKS}*{Equations 4.8.5, 4.8.6}. The
proof of part (a) is complete.

As for part (b), 
using Proposition \ref{prop:1}, $\lim_{\b\to \infty}\bbP_{0,L}(\O^{{\rm reg}}_{0,L})=1$ uniformly in $L$.  Hence it is enough to prove the statement of the lemma for
$\bbP^{{\rm reg}}_{0,L}(\cdot):=\bbP_{0,L}(\cdot\mid \O_{0,L}^{\rm reg})$, with $3/4$
replaced by e.g. $3/5$. By writing as before $\g\in \O^{{\rm reg}}_{0,L}$
as the concatenation of two
independent paths $\g_1,\g_2$ we get that 
\[
\bbP^{{\rm reg}}_{0,L}(\g_{\rm max}(L/2)\geq -j)= \hat\cP_{n}\otimes
\hat\cP_{n} (h_n(\g_1)\geq -j \tc
h_n(\g_1)=h_n(\g_2)).
\]
At this stage we proceed exactly as in part (b) of Proposition
\ref{prop:2} (cf. \eqref{eq:10}). Using again Lemma \ref{lemma:DKS}, we get that,
for any small $\d$ and any $L\geq Ce^{\b}$ with $C$ large enough depending on $\d$,
\[
 \frac{\int_{-\pi \s_{n}}^{\pi\s_{n}} dt\  |\hat f_{n}(t)|}{\int_{-\pi \s_{n}}^{\pi\s_{n}}dt\
  |\hat f_{n}(t)|^2}\leq  \sqrt{2}+\d.
\]
As far as the quantity $\hat \cP_{n}(h(\g_1)\geq -j)$ is
concerned we can appeal to the following
formula (cf. e.g. \cite{Feller}*{Ch. XVI formula 3.13}):
\[
\max_{h}|\hat \cP_{n }(h(\g)\geq h\s_{n})-\Phi(-h)|\le \frac 1\pi \int_{-T}^T dt\
\frac 1t |\hat f_{n }(t)-e^{-t^2/2}| +\frac{24}{\sqrt{2\pi^3}T},\quad
\forall T>0,
\]
valid for all $h$. Choose now $T=\frac{48}{\sqrt{2\pi^3}\d}$ and a constant $C$ large
enough depending on $\d$ in a such a way that  $\int_{-T}^T dt\
\frac 1t |\hat f_{n }(t)-e^{-t^2/2}|\le \d/2$ for all $L\geq C e^{\b}$. Thus the r.h.s. above
is smaller than $\d$. Taking $h=-j/\si_n$ concludes the proof.
\end{proof}
We end with a simple lemma that allows one to bound canonical probabilities with their
grand canonical counterpart. We use the notation introduced so far.
\begin{lemma}
\label{lem:equiv}
Fix $k\in \bbN,\ L\geq k$ and let $A\subset \O_{0,L}$ depend only on the first
$k$ edges of the
path $\g$. If $\lim_{\b\to \infty}\sup_{L\geq k}\cP_L(A)=0$ then the
same holds for the canonical measure $\bbP_{0,L}$.
\end{lemma}
\begin{proof}
Fix a large constant $C$. If $L\le Ce^{\b}$ then it is immediate to
check that, as $\b\to \infty$, the measure $\bbP_{0,L}$ is
concentrated over paths which are flat for the first $k$ steps as is
the case for the grand canonical measure, so the two marginals on the
first $k$ steps are essentially identical. If
instead $L\geq Ce^{\b}$ we argue as in part (b) of Proposition
\ref{prop:3} 
to write (with $n=\lfloor L\rfloor$)
\begin{align*}
\lim_{\b\to \infty}\sup_{L\geq k}\bbP_{0,L}(A)&=\lim_{\b\to
  \infty}\sup_{L\geq k}\bbP^{\rm
  reg}_{0,L}(A)=\lim_{\b\to \infty}\sup_{2n\geq k}\hat\cP_n\otimes\hat\cP_n\left(A\mid
  h_n(\g_1)=h_n(\g_2)\right)\\
&\leq  \lim_{\b\to \infty}\sup_{2n\geq k}\frac{\int_{-\pi \s_{n}}^{\pi\s_{n}} dt\  |\hat f_{n}(t)|}{\int_{-\pi \s_{n}}^{\pi\s_{n}}dt\
  |\hat f_{n}(t)|^2}\,\hat\cP_n(A)\le (\sqrt{2}+\d) \lim_{\b\to \infty}\sup_{L\geq k}\,\hat \cP_L(A),
  \end{align*}
where $\d=\d(C)$ tends to zero as $C\to +\infty$. Finally we observe
that $$\lim_{\b\to \infty}\sup_{L\geq k}\hat \cP_L(A)=\lim_{\b\to \infty}\sup_{L\geq k}\cP_L(A)=0.$$

\end{proof}


\begin{bibdiv}
\begin{biblist}

%
%

\bib{Bargman}{article}{
    AUTHOR = {Bargmann, Valentine},
     TITLE = {On the number of bound states in a central field
of force},
   JOURNAL = {Proc Natl Acad Sci U S A.},
    VOLUME = {38},
      YEAR = {1952},
    NUMBER = {11},
     PAGES = {961--966},
}


\bib{CV}{article}
{
    AUTHOR = {Caputo, P.},
    AUTHOR={Velenik, Y.},
     TITLE = {A note on wetting transition for gradient fields},
   JOURNAL = {Stochastic Process. Appl.},
      VOLUME = {87},
      date = {2000},
    NUMBER = {1},
     PAGES = {107--113},
}

\bib{CLMST}{article}
 {
  author = {Caputo, Pietro},
  author = { Lubetzky, Eyal},
  author={ Martinelli, Fabio},
  author={ Sly, Allan},
  author={Toninelli, Fabio Lucio},
  title =		 {Scaling limit and cube-root fluctuations in SOS surfaces above a wall},
  year =		 {2013},
  note =		 {To appear on J. Eur. Math. Soc., preprint \texttt{arXiv:1302.6941}},
}

\bib{CMT}{article}
 { 
  author =		 {Caputo, Pietro},
  author={ Martinelli, Fabio},
  author={Toninelli, Fabio Lucio},
 title =		 {On the probability of staying above a wall for the (2+1)-dimensional SOS model at low temperature},
  year =		 {2014},
  note =		 {preprint \texttt{arXiv:1406.1206}},
}

\bib{Caravennaet}{article}{
AUTHOR = {Caravenna, F.},
AUTHOR={P\'etr\'elis N.},
TITLE={A polymer in a multi-interface medium},
JOURNAL={Ann. Appl. Probab.},
VOLUME={19},
YEAR={2009}, 
PAGES={1803-1839},
}



\bib{CGZ}{article}{
    AUTHOR = {Caravenna, F.},
    AUTHOR={Giacomin, G.},
    AUTHOR={Zambotti, Lorenzo},
     TITLE = {Sharp asymptotic behavior for wetting models in
              {$(1+1)$}-dimension},
   JOURNAL = {Electron. J. Probab.},
    VOLUME = {11},
      YEAR = {2006},
     PAGES = {no. 14, 345--362 (electronic)},
}




\bib{DGZ}{article}{
    AUTHOR = {Deuschel, Jean-Dominique},
    AUTHOR={Giacomin, G.},
    AUTHOR={Zambotti, Lorenzo}, 
         TITLE = {Scaling limits of equilibrium wetting models in
              {$(1+1)$}-dimension},
   JOURNAL = {Probab. Theory Related Fields},
    VOLUME = {132},
      YEAR = {2005},
    NUMBER = {4},
     PAGES = {471--500},
}

\bib{DKS}{book}{
   author={Dobrushin, R.},
   author={Koteck{\'y}, R.},
   author={Shlosman, S.},
   title={Wulff construction. A global shape from local interaction},
   series={Translations of Mathematical Monographs},
   volume={104},
   publisher={American Mathematical Society},
   place={Providence, RI},
   date={1992},
   pages={x+204},
}

\bib{Doney}{article}{
   AUTHOR = {Doney, R. A.},
     TITLE = {Local behaviour of first passage probabilities},
   JOURNAL = {Probab. Theory Related Fields},
     VOLUME = {152},
      YEAR = {2012},
    NUMBER = {3-4},
     PAGES = {559--588},
 }
 
 \bib{Feller}{book}{
   author={Feller, W.},
   title={An introduction to probability theory and its applications, Vol. II},
 publisher={Second edition. John Wiley \& Sons Inc.},
   place={New York},
   date={1971},
}

\bib{Fisher}{article}{
   author={Fisher, Michael E.},
   title={Walks, walls, wetting, and melting},
   journal={J. Statist. Phys.},
   volume={34},
   date={1984},
   number={5-6},
   pages={667--729},
}

\bib{Giacomin}{book}{
   author={Giacomin, Giambattista},
   title={Random polymer models},
   publisher={Imperial College Press},
   place={London},
   date={2007},
   pages={xvi+242},
}


\bib{IST}{article}
 { 
  author =		 {Ioffe, Dima},
  author={Shlosman, Senya},
  author={Toninelli, Fabio Lucio},
 title =		 {Interaction versus entropic repulsion for low temperature Ising polymers},
  year =		 {2014},
  note =		 {preprint \texttt{arXiv:1407.3592}},
}


\bib{IsoYos}{article}{
    AUTHOR = {Isozaki, Yasuki},
    AUTHOR={ Yoshida, Nobuo},
     TITLE = {Weakly pinned random walk on the wall: pathwise descriptions
              of the phase transition},
   JOURNAL = {Stochastic Process. Appl.},
    VOLUME = {96},
      YEAR = {2001},
    NUMBER = {2},
     PAGES = {261--284},
}

\bib{Jost}{article}{
    AUTHOR = {Jost, Res},
AUTHOR={Pais, Abraham},
     TITLE = {On the Scattering of a Particle by a Static Potential},
   JOURNAL = {Phys. Rev.},
    VOLUME = {82},
      YEAR = {1951},
     PAGES = {840},
}

\bib{Sohier}{article}{
    AUTHOR = {Sohier, Julien},
     TITLE = {The scaling limits of a heavy tailed {M}arkov renewal process},
   JOURNAL = {Ann. Inst. Henri Poincar\'e Probab. Stat.},
    VOLUME = {49},
      YEAR = {2013},
    NUMBER = {2},
     PAGES = {483--505},
    }
\bib{Solomyak}{article}{
    AUTHOR = {Solomyak, Michael},
     TITLE = {On a class of spectral problems on the half-line and their
              applications to multi-dimensional problems},
   JOURNAL = {J. Spectr. Theory},
    VOLUME = {3},
      YEAR = {2013},
    NUMBER = {2},
     PAGES = {215--235},
}

\end{biblist}
\end{bibdiv}
\end{document}